\documentclass[11pt]{amsart}

\usepackage{amsmath,amssymb}
\usepackage{graphicx,epsfig,color}

\newtheorem{theorem}{Theorem}
\newtheorem{lemma}{Lemma}
\newtheorem{corollary}{Corollary}
\newtheorem{proposition}{Proposition}
\newtheorem{remark}{Remark}

\title[A nonlinear crack - defect interaction problem]{
A shape-topological control problem for nonlinear crack - defect interaction:\\   
the anti-plane variational model}

\author{Victor A. Kovtunenko$^1$}

\thanks{$^1$ Institute for Mathematics and Scientific Computing, 
Karl-Franzens University of Graz, NAWI Graz, Heinrichstr.36, 8010 Graz, Austria; 
and
Lavrent'ev Institute of Hydrodynamics, 630090 Novosibirsk, Russia
(E-mail: victor.kovtunenko@uni-graz.at)}

\author{G\"unter Leugering$^2$}

\thanks{$^2$ Applied Mathematics 2, 
Friedrich-Alexander University of Erlangen-N\"urnberg, 
Cauerstr.11, 91058 Erlangen, Germany
(E-mail: leugering@math.fau.de)}

\subjclass{35B25, 49J40, 49Q10, 74G70.}

\keywords{Shape-topological control, topological derivative, singular perturbation, 
variational inequality, crack-defect interaction, nonlinear crack with non-penetration, 
anti-plane stress intensity factor, strain energy release rate, dipole tensor}

\begin{document}

\begin{abstract}
We consider the shape-topological control of a singularly perturbed variational inequality. 
The geometry-dependent state problem that we address in this paper concerns 
a heterogeneous medium with a micro-object (defect) and a macro-object (crack) 
modeled in 2d. 

The corresponding nonlinear optimization problem subject to inequality constraints at
the crack is considered within a general variational framework. 
For the reason of asymptotic analysis, singular perturbation theory is applied
resulting in the topological sensitivity of an objective function representing the
release rate of the strain energy. 
In the vicinity of the nonlinear crack, the anti-plane strain energy release rate 
is expressed by means of the mode-III stress intensity factor, that is
examined with respect to small defects like micro-cracks, holes, and inclusions 
of varying stiffness. 
The result of shape-topological control is useful either for arrests or rise 
of crack growth. 
\end{abstract}

\maketitle

\section{Introduction}\label{sec0}

The paper aims at shape-topological control of geometry-dependent variational 
inequalities, which are motivated by application to non-linear cracking phenomena. 

From a physical point of view, both cracks and defects appear in heterogeneous 
media and composites in the context of fracture. 
We refer to \cite{MP/00} for a phenomenological approach to fracture with and
without defects. 
Particular cases for the linear model of a stress-free crack interacting with
inhomogeneities and micro-defects were considered in
\cite{Gon/95,MK/01,PMMM/12}. 
In the present paper we investigate the sensitivity of a nonlinear crack 
with respect to a small object (called defect) of arbitrary physical and
geometric nature. 

While the classic model of a crack is assumed linear, the physical consistency
needs nonlinear modeling. 
Nonlinear crack models subject to non-penetration (contact) conditions have
been developed in \cite{FHLRS/09,KK/00,KS/99,KS/00,Kov/01,KK/07} 
and other works by the authors. 
Recently, nonlinear cracks were bridged with thin inclusions under non-ideal
contact, see \cite{IKRT/12,KL/10,KLSN/12}. 
In the present paper we confine ourselves to the anti-plane model
simplification, in which case inequality type constraints at
the plane crack are argued in \cite{KKT/10,KK/08}. 
The linear crack is included there as the particular case. 

From a mathematical point of view, a topology perturbation problem is
considered by varying defects posed in a cracked domain. 
For shape and topology optimization of cracks we refer to \cite{AKLP/10,AHM/05,FV/89} 
and to \cite{SZ/92} for shape perturbations in a general context. 
As the size of the defect tends to zero, we have to employ singular perturbation theory. 
The respective asymptotic methods were developed in
\cite{AJT/04,Il'/92,MNP/00}, mostly for linear partial differential
equations (PDE) stated in singularly perturbed domains. 
Nevertheless, nonlinear boundary conditions are admissible to impose at those
boundaries which are separated from the varying object, as it is described in
\cite{AS/03,FLSS/07}. 

From the point of view of shape and topology optimization, we investigate a
novel setting of interaction problems between dilute geometric objects. 
In a broad scope, we consider a new class of geometry-dependent
objective functions $J$ which are perturbed by at least two interacting objects
$\Gamma$ and $\omega$ such that 
\begin{equation*}
J: \{\Gamma\}\times\{\omega\}\mapsto\mathbb{R},\quad J=J(\Gamma,\omega). 
\end{equation*}
In particular, we look how a perturbation of one geometric object, say $\omega$,
will affect a topology sensitivity, here the derivative of $J$ with respect to
another geometric object $\Gamma$. 
In our particular setting of the interaction problem 
the symbol $\Gamma$ refers to a crack
and $\omega$ to an inhomogeneity (defect) in a heterogeneous medium. 

The principal difficulty is that $\Gamma$ and $\omega$ enter the objective $J$ in 
a fully implicit way through a solution of a state (PDE) geometry-dependent problem. 
Therefore, to get an 
explicit formula, we rely on asymptotic modeling concerning the smallness of  $\omega$. 
Moreover, we generalize the state problem by allowing it to be a variational inequality. 
In fact, the variational approach to the perturbation problem allows us to
incorporate nonlinear boundary conditions stated at the crack $\Gamma$. 

The outline of the paper is as follows. 

To get an insight into the mathematical problem, 
in Section~\ref{sec1} we start with a general concept of shape-topological 
control for singular perturbations of abstract variational inequalities. 
In Sections~\ref{sec2} and \ref{sec3} this concept is specified 
for the nonlinear dipole problem of crack-defect interaction in 2d. 

For the anti-plane model introduced in Section~\ref{sec2}, further in Section~\ref{sec3} 
we provide the topological sensitivity of an objective function expressing 
the strain energy release rate $J_{\rm SERR}$ by means of the mode-III stress 
intensity factor $J_{\rm SIF}$ which is of primary importance for engineers. 
The first order asymptotic term determines the so-called {\em topological
derivative} of the objective function with respect to diminishing defects like
holes and inclusions of varying stiffness. 
We prove its semi-analytic expression by using a dipole representation of the
crack tip - the defect center with the help of a Green type (weight) function. 
The respective dipole matrix is related inherently to polarization and virtual
mass matrices, see \cite{PS/51}. 

Within an equivalent ellipse concept, see for example
\cite{BHV/03,PMMM/12}, we further derive explicit formulas of the
dipole matrix for the particular cases of the ellipse shaped defects. 
Holes and rigid inclusions are accounted here as the two limit cases of the
stiffness parameter $\delta\searrow+0$ and $\delta\nearrow\infty$, respectively 
(see Appendix~\ref{A}). 

The asymptotic result of shape-topological control is useful to force either 
shielding or amplification of an incipient crack by posing trial inhomogeneities 
(defects) in the test medium. 

\section{Shape-topological control}\label{sec1}

In the abstract context of shape-topological differentiability, 
see e.g. \cite{LSZ/14,LSZ/15}, our construction can be outlined as follows. 

We deal with variational inequalities of the type: 
Find $u^0\in K$ such that
\begin{equation}\label{1.1}
\begin{split}
&\langle A u^0 -G, v-u^0\rangle\ge0\quad\text{for all $v\in K$}
\end{split}
\end{equation}
with a linear strongly monotone operator $A: H\mapsto H^\star$, 
fixed $G\in H^\star$, and a polyhedric cone $K\subset H$, 
which are defined in a Hilbert space $H$ and its dual space $H^\star$. 
The solution of variational inequality \eqref{1.1} implies 
a metric projection $P_K: H^\star\mapsto K$, $G\mapsto u^0$. 
Its differentiability properties are useful in control theory, see \cite{LSZ/14,LSZ/15}. 

For control in the 'right-hand side' (the inhomogeneity) of \eqref{1.1}, 
one employs {\em regular perturbations} of $G$ 
with a small parameter $\varepsilon>0$ in the direction of $h\in H^\star$: 
Find $u^\varepsilon\in K$ such that
\begin{equation}\label{1.2}
\begin{split}
&\langle A u^\varepsilon -(G +\varepsilon h), v-u^\varepsilon\rangle\ge0
\quad\text{for all $v\in K$}.
\end{split}
\end{equation}
Then the directional differentiability of $P_K(G +\varepsilon h)$ from the right as 
$\varepsilon=+0$ implies the following linear asymptotic expansion 
\begin{equation}\label{1.3}
\begin{split}
&u^\varepsilon =u^0 +\varepsilon q +{\rm o}(\varepsilon)
\;\text{in $H$ as $\varepsilon\searrow+0$}
\end{split}
\end{equation}
with $q\in  S(u^0)$ uniquely determined on a proper convex cone $S(u^0)$, 
$K\subset S(u^0)\subset H$, and depending on $u^0$ and $h$, 
see \cite{LSZ/14,LSZ/15} for details. 

In contrast, our underlying problem implies {\em singular perturbations} 
and the control of the operator $A$ of \eqref{1.1}, namely: 
Find $u^\varepsilon\in K$ such that
\begin{equation}\label{1.4}
\begin{split}
&\langle A_\varepsilon u^\varepsilon -G, v-u^\varepsilon\rangle\ge0
\quad\text{for all $v\in K$},
\end{split}
\end{equation}
where $A_\varepsilon =A +\varepsilon F_\varepsilon$, with a
bounded linear operator $F_\varepsilon: H\mapsto H^\star$ such that 
$A_\varepsilon$ is strongly monotone, uniformly in $\varepsilon$, 
and $\varepsilon \|F_\varepsilon\|={\rm O}(\varepsilon)$. 
In this case, we arrive at the nonlinear  representation in $\varepsilon\searrow+0$
\begin{equation}\label{1.5}
\begin{split}
&u^\varepsilon =u^0 +\varepsilon \tilde{q}^\varepsilon +{\rm O}(f(\varepsilon))
\;\text{in $H$},\quad \varepsilon \|\tilde{q}^\varepsilon\| ={\rm O}(\varepsilon).  
\end{split}
\end{equation}
In \eqref{1.5} $\tilde{q}^\varepsilon$ depends on $u^0$ and $F_\varepsilon$. 
A typical example, $\tilde{q}^\varepsilon(x) 
=\tilde{q}\bigl({\textstyle\frac{x}{\varepsilon}}\bigr)$, implies the existence of a boundary layer, 
e.g. in homogenization theory. 
In contrast to the differential $q$ in \eqref{1.3}, a representative 
$\varepsilon \tilde{q}^\varepsilon$ is not uniquely defined by $\varepsilon$ 
but depends also on ${\rm o}(f(\varepsilon))$-terms. 
Examples are slant derivatives. 
The asymptotic behavior $f(\varepsilon)$ of the residual in \eqref{1.5} 
may differ for concrete problems. 
Thus, in the subsequent analysis $f(\varepsilon)=\varepsilon^2$ in 2d. 

In order to find the representative $\tilde{q}^\varepsilon$ in \eqref{1.5}, 
we suggest sufficient conditions \eqref{1.6}--\eqref{1.9} below.
\begin{proposition}\label{prop1.1}
If the following relations hold: 
\begin{eqnarray}
&u^0 +\epsilon \tilde{q}^\varepsilon\in K,\label{1.6}\\
&u^\varepsilon-\varepsilon \tilde{q}^\varepsilon\in K,\label{1.7}\\
&\langle A_\varepsilon \tilde{q}^\varepsilon +F_\varepsilon u^0 -R_\varepsilon, v\rangle =0
\quad\text{for all $v\in H$},\label{1.8}\\
&\varepsilon\|R_\varepsilon\| ={\rm O}(f(\varepsilon)),\label{1.9}
\end{eqnarray}
then \eqref{1.5} holds for the solutions of variational inequalities \eqref{1.1} and \eqref{1.4}. 
\end{proposition}

\begin{proof}
Indeed, plugging test functions $v=u^\varepsilon-\varepsilon \tilde{q}^\varepsilon\in K$ 
in \eqref{1.1} due to \eqref{1.7} and $v=u^0 +\varepsilon \tilde{q}^\varepsilon\in K$ in \eqref{1.4} due to \eqref{1.6}, after summation 
\begin{equation*}
\begin{split}
&\langle A_\varepsilon (u^\varepsilon -u^0) +\varepsilon F_\varepsilon u^0, 
u^\varepsilon -u^0 -\varepsilon \tilde{q}^\varepsilon\rangle\le0, 
\end{split}
\end{equation*} 
and substituting $v=u^\varepsilon -u^0 -\varepsilon \tilde{q}^\varepsilon$ in \eqref{1.8} 
multiplied by $-\varepsilon$, this yields 
\begin{equation*}
\begin{split}
&\langle A_\varepsilon (u^\varepsilon -u^0 -\varepsilon \tilde{q}^\varepsilon) 
+\varepsilon R_\varepsilon,
u^\varepsilon -u^0 -\varepsilon \tilde{q}^\varepsilon\rangle\le0. 
\end{split}
\end{equation*}
Applying here the Cauchy--Schwarz inequality together with \eqref{1.9} it 
shows \eqref{1.5} and completes the proof. 
\end{proof}

Our consideration aims at shape-topological control by means of 
mathematical programs with equilibrium constraints (MPEC): 
Find optimal parameters $p\in P$ from a feasible set $P$ such that
\begin{equation}\label{1.10}
\begin{split}
&\underset{p\in P}{\rm minimize}
\,J(u^{(\varepsilon,p)}) 
\quad\text{subject to $\Pi(u^{(\varepsilon,p)}) =\min_{v\in K_p} \Pi(v)$}.
\end{split}
\end{equation}
In \eqref{1.10} the functional $\Pi:H\mapsto\mathbb{R}$,  
$\Pi(v) :=\langle {\textstyle\frac{1}{2}} A_\varepsilon v -G, v\rangle$ 
represents the strain energy (SE) of the state problem, 
such that variational inequality \eqref{1.4} implies the first order optimality 
condition for the minimization of $\Pi(v)$ over $K_p\subset H$. 
The multi-parameter $p$ may include the right-hand side 
$G$, geometric variables, and other data of the problem. 
The optimal value function $J$ in \eqref{1.10} is motivated by underlying physics, 
which we will specify in examples below.

The main difficulty of the shape-topological control is that geometric parameters 
are involved in MPEC in fully implicit way. 
In this respect, relying on asymptotic models under small variations $\varepsilon$ 
of geometry is helpful to linearize the optimal value function. 
See e.g. the application of topological sensitivity to inverse scattering 
problems in \cite{KK/14}. 

In order to expand \eqref{1.10} in $\varepsilon\searrow+0$,  the uniform 
asymptotic expansion \eqref{1.5} is useful which, however, is varied by $f(\varepsilon)$. 
The variability of $f(\varepsilon)$ is inherent here due to non-uniqueness 
of a representative $\varepsilon\tilde{q}^\varepsilon$  
defined up to the ${\rm o}(f(\varepsilon))$-terms. 
As an alternative, developing variational technique related to Green functions 
and truncated Fourier series, in Section~\ref{sec3.2} we derive 
local asymptotic expansions in the near-field, which are uniquely determined.

Since Proposition~\ref{prop1.1} gives only sufficient conditions for \eqref{1.5}, 
in the following sections we suggest a method of topology perturbation 
finding the correction $\tilde{q}^\varepsilon$ for the underlying 
variational inequality. 

\section{Nonlinear problem of crack-defect interaction in 2d}\label{sec2}

We start with the 2d-geometry description. 

\subsection{Geometric configuration}\label{sec2.1}

For $x=(x_1,x_2)^\top\in\mathbb{R}^2$ we set the semi-infinite straight crack
$\Gamma_\infty =\{x\in\mathbb{R}^2: x_1<0, x_2=0\}$ with 
the unit normal vector $n=(0,1)^\top$ at $\Gamma_\infty$. 
Let $\Omega\subset\mathbb{R}^2$ be a bounded domain with the Lipschitz boundary
$\partial\Omega$ and the normal vector $n=(n_1,n_2)^\top$ at $\partial\Omega$. 
We assume that the origin $0\in\Omega$ and assign it to the tip of a finite crack
$\Gamma :=\Gamma_\infty\cap\Omega$. 
An example geometric configuration is drawn in Figure~\ref{fig_dipole}.
\begin{figure}[hbt!]
\begin{center}
\epsfig{file=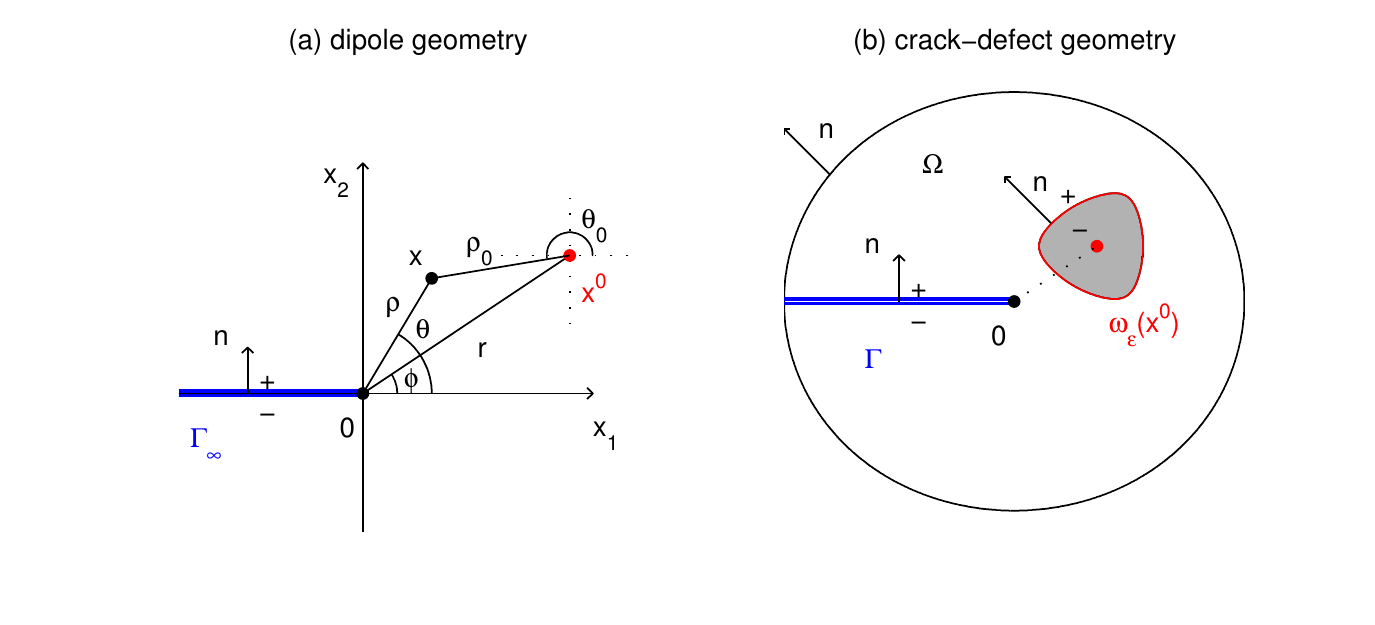,height=6cm,angle=0}
\caption{Example geometric configuration.}
\label{fig_dipole}
\end{center}
\end{figure}
%

Let $x^0$ be an arbitrarily fixed point in the cracked domain
$\Omega\setminus\Gamma$. 
We associate the poles $0$ and $x^0$ with two polar coordinate systems
$x=\rho(\cos\theta,\sin\theta)^\top$, $\rho>0$, $\theta\in[-\pi,\pi]$, and
$x-x^0=\rho_0(\cos\theta_0,\sin\theta_0)^\top$, $\rho_0>0$,
$\theta_0\in(-\pi,\pi]$. 
Here $x^0=r(\cos\phi,\sin\phi)^\top$ is given by $r>0$ and $\phi\in(-\pi,\pi)$
as it is depicted in Figure~\ref{fig_dipole} (a).
We refer $x^0$ to the center of a defect
$\omega_\varepsilon(x^0)$ posed in $\Omega$ as illustrated in
Figure~\ref{fig_dipole} (b). 

More precisely, let a trial geometric object be given by the compact set
$\omega_\varepsilon(x^0)=\{x\in\mathbb{R}^2:
\frac{x-x^0}{\varepsilon}\in\omega\}$ which is parametrized by an admissible
triple of the shape $\omega\in\Theta$, center $x^0\in\Omega\setminus\Gamma$, 
and size $\varepsilon>0$. 
Let $B_\rho(x^0)$ denote the disk around $x^0$ of radius $\rho$. 
For admissible shapes $\Theta$ we 
choose a domain $\omega$ such that $0\in\omega\subseteq B_1(0)$ 
and $\rho=1$ is the minimal radius among all bounding discs 
$B_\rho(0)\supset\omega$. 
Thus, the shapes are invariant to translations and isotropic scaling, so that
we express them with the equivalent notation $\omega=\omega_1(0)$. 
Admissible geometric parameters $(\omega,\varepsilon,x^0)
\in\Theta\times\mathbb{R}_+\times(\Omega\setminus\Gamma)$ 
should satisfy the consistency condition
$\omega_\varepsilon(x^0) \subset B_\varepsilon(x^0)
\subset\Omega\setminus\Gamma$. 

We note that the motivation of inclusion $\omega\subseteq B_1(0)$ 
(but not $\omega\supseteq B_1(0)$) is to separate the far-field 
$\mathbb{R}^2\setminus B_1(0)$ from the near-field $B_1(0)\setminus\omega$ 
of the object $\omega$.

In the following we assume that the Hausdorff measure ${\rm meas}_2(\omega)>0$, 
the boundary$\partial\omega_\varepsilon(x^0)$ is Lipschitz continuous 
and assign $n$ to the unit normal vector at $\partial\omega_\varepsilon(x^0)$ 
which points outward to $\omega_\varepsilon(x^0)$. 
In a particular situation, our consideration admits also the
degenerate case when $\omega_\varepsilon(x^0)$ shrinks to a 1d
Lipschitz manifold of co-dimension one in $\mathbb{R}^2$, thus, 
allowing for defects like curvilinear inclusions. 
The degenerate case will appear in more detail when shrinking ellipses to
line segments as described in Appendix~\ref{A}. 

\subsection{Variational problem}\label{sec2.2}

In the reference configuration of the cracked domain $\Omega\setminus\Gamma$
with the fixed inclusion $\omega_\varepsilon(x^0)$ we state a constrained
minimization problem related to PDE, here, a model problem with the scalar
Laplace operator. 
Motivated by 3d-fracture problems with possible contact between crack faces, as described in \cite{KKT/10}, in the anti-plane framework of linear
elasticity, we look for admissible displacements $u(x)$ in
$\Omega\setminus\Gamma$ which are restricted along the crack by the inequality
constraint 
\begin{equation}\label{2.1}
[\![u]\!] =u|_{\Gamma^+} -u|_{\Gamma^-}\ge0 \quad\text{on $\Gamma$}.
\end{equation}
The positive $\Gamma^+_\infty$ (hence, its part
$\Gamma^+=\Gamma^+_\infty\cap\Omega$) and the negative $\Gamma^-_\infty$ (hence,
$\Gamma^-=\Gamma^-_\infty\cap\Omega$) crack faces are distinguished as the
limit of points $(x_1,x_2)^\top$ for $x_1<0$ and $x_2\to0$, when $x_2>0$ and
$x_2<0$, respectively, see Figure~\ref{fig_dipole}. 

Now we get a variational formulation of a state problem due to the
unilateral constraint \eqref{2.1}. 

Let the external boundary $\partial\Omega$ consist of two disjoint parts
$\Gamma_N$ and $\Gamma_D$. 
We assume that the Dirichlet part has the positive measure 
${\rm meas}_1(\Gamma_D)>0$, 
otherwise we should exclude the nontrivial kernel (the rigid displacements) 
for coercivity of the objective functional $\Pi$ in \eqref{2.3} below. 
The set of admissible displacements contains functions $u$ from the Sobolev space 
\begin{equation*}
H(\Omega\setminus\Gamma) =\{u\in H^1(\Omega\setminus\Gamma):\;
u=0\quad\text{on $\Gamma_D$}\}
\end{equation*}
such that \eqref{2.1} holds:
\begin{equation*}
K(\Omega\setminus\Gamma) =\{u\in H(\Omega\setminus\Gamma):\;
[\![u]\!]\ge0\quad\text{on $\Gamma$}\}.
\end{equation*}
This is a convex cone in $H(\Omega\setminus\Gamma)$, 
moreover, a polyhedric cone, see \cite{LSZ/14,LSZ/15}. 
We note that the jump of the traces at $\Gamma$ is defined well in the
Lions--Magenes space $[\![u]\!]\in H^{1/2}_{00}(\Gamma)$, see \cite[Section~1.4]{KK/00}. 

Let $\mu>0$ be a fixed material parameter (the Lame constant) in the 
homogeneous reference domain $\Omega\setminus\Gamma$. 
We distinguish the inhomogeneity with the help of a variable parameter
$\delta>0$, such that the characteristic function is given by
\begin{equation}\label{2.2}
\chi^\delta_{{}_{\omega_\varepsilon(x^0)}}(x) :=1
-(1-\delta) \mathbf{1}_{{}_{\omega_\varepsilon(x^0)}}
=\left\{\begin{array}{ll} 1&,\;x\in\Omega\setminus\omega_\varepsilon(x^0)\\
\delta&,\;x\in\omega_\varepsilon(x^0) \end{array}\right..
\end{equation}
In the following we use the notation $\mu
\chi^\delta_{{}_{\omega_\varepsilon(x^0)}}$, which implies, due to \eqref{2.2},
the material parameter $\mu$ in the homogeneous domain
$\Omega\setminus\omega_\varepsilon(x^0)$, and the material parameter $\mu\delta$
in $\omega_\varepsilon(x^0)$ characterizing stiffness of the inhomogeneity. 
The parameter $\delta$ accounts for three physical situations: 
inclusions of varying stiffness for finite $0<\delta<\infty$, 
holes for $\delta\searrow+0$, and rigid inclusions for $\delta\nearrow+\infty$. 

For given boundary traction $g\in L^2(\Gamma_N)$, the strain energy 
of the heterogeneous medium is described by the functional 
$\Pi: H(\Omega\setminus\Gamma)\mapsto\mathbb{R}$,
\begin{equation}\label{2.3}
\begin{split}
&\Pi(u;\Gamma,\omega_\varepsilon(x^0)) :={\textstyle\frac{1}{2}}
\int_{\Omega\setminus\Gamma} \mu \chi^\delta_{{}_{\omega_\varepsilon(x^0)}}
|\nabla u|^2\,dx -\int_{\Gamma_N} gu\,dS_x, 
\end{split}
\end{equation}
which is quadratic and strongly coercive over $H(\Omega\setminus\Gamma)$. 
Henceforth, the Babu\v{s}ka--Lax--Milgram theorem guarantees the unique
solvability of the constrained minimization of $\Pi$ over
$K(\Omega\setminus\Gamma)$, which implies the variational formulation of the
{\em heterogeneous problem}: Find $u^{(\omega,\varepsilon,x^0,\delta)}\in
K(\Omega\setminus\Gamma)$ such that 
\begin{equation}\label{2.4}
\begin{split}
&\int_{\Omega\setminus\Gamma} \mu \chi^\delta_{{}_{\omega_\varepsilon(x^0)}}
(\nabla u^{(\omega,\varepsilon,x^0,\delta)})^\top \nabla(v
-u^{(\omega,\varepsilon,x^0,\delta)})\,dx\\
&\ge\int_{\Gamma_N} g(v -u^{(\omega,\varepsilon,x^0,\delta)}) \,dS_x
\quad\text{for all $v\in K(\Omega\setminus\Gamma)$}.
\end{split}
\end{equation}

The variational inequality \eqref{2.4} describes the weak solution of the
following boundary value problem: 
\begin{subequations}\label{2.5}
\begin{equation}\label{2.5a}
-\Delta u^{(\omega,\varepsilon,x^0,\delta)} =0\quad\text{in
$\Omega\setminus\Gamma$},
\end{equation}
\begin{equation}\label{2.5b}
u^{(\omega,\varepsilon,x^0,\delta)} =0\quad\text{on
$\Gamma_D$},\quad \mu {\textstyle\frac{\partial
u^{(\omega,\varepsilon,x^0,\delta)}}{\partial n}}
=g\quad\text{on $\Gamma_N$},
\end{equation}
\begin{equation}\label{2.5c}
\begin{split}
&\bigl[\!\!\bigl[ {\textstyle\frac{\partial
u^{(\omega,\varepsilon,x^0,\delta)}}{\partial n}} \bigr]\!\!\bigr] =0,\quad
[\![u^{(\omega,\varepsilon,x^0,\delta)}]\!]\ge0,\quad {\textstyle\frac{\partial
u^{(\omega,\varepsilon,x^0,\delta)}}{\partial n}}\le0,\\
&{\textstyle\frac{\partial u^{(\omega,\varepsilon,x^0,\delta)}}{\partial n}}
[\![u^{(\omega,\varepsilon,x^0,\delta)}]\!] =0
\quad\text{on $\Gamma$},
\end{split}
\end{equation}
\begin{equation}\label{2.5d}
\begin{split}
&{\textstyle\frac{\partial
u^{(\omega,\varepsilon,x^0,\delta)}}{\partial
n}}|_{{}_{\partial\omega_\varepsilon(x^0)^+}}
-\delta {\textstyle\frac{\partial
u^{(\omega,\varepsilon,x^0,\delta)}}{\partial
n}}|_{{}_{\partial\omega_\varepsilon(x^0)^-}} =0,\\
&[\![u^{(\omega,\varepsilon,x^0,\delta)}]\!]=0
\quad\text{on $\partial\omega_\varepsilon(x^0)$}.
\end{split}
\end{equation}
\end{subequations}
In \eqref{2.5d} the jump across the defect boundary is defined as 
\begin{equation}\label{2.6}
[\![u]\!] =u|_{{}_{\partial\omega_\varepsilon(x^0)^+}}
-u|_{{}_{\partial\omega_\varepsilon(x^0)^-}} \quad\text{on
$\partial\omega_\varepsilon(x^0)$},
\end{equation}
where $+$ and $-$ correspond to the chosen direction of the normal $n$, which
is outward to $\omega_\varepsilon(x^0)$, see Figure~\ref{fig_dipole} (b). 

We remark that the $L^2$-regularity of the normal derivatives at the boundaries 
$\Gamma_N$, $\Gamma$, and $\partial\omega_\varepsilon(x^0)$ 
is needed in order to have strong solutions in \eqref{2.5}. 
The exact sense to the boundary conditions \eqref{2.5c} can be given for the
traction $\frac{\partial u^{(\omega,\varepsilon,x^0,\delta)}}{\partial n}$ in
the dual space of $H^{1/2}_{00}(\Gamma)$, 
which is denoted by $H^{1/2}_{00}(\Gamma)^\star$, 
to \eqref{2.5b} for $\frac{\partial u^{(\omega,\varepsilon,x^0,\delta)}}{\partial n}$ 
in the dual space of $H^{1/2}_{00}(\Gamma_N)$, 
and to \eqref{2.5d} for $\frac{\partial u^{(\omega,\varepsilon,x^0,\delta)} 
}{\partial n} |_{{}_{\partial\omega_\varepsilon(x^0)^\pm}} 
\in H^{-1/2}(\partial\omega_\varepsilon(x^0))$.  
Moreover, the solution $u^{(\omega,\varepsilon,x^0,\delta)}$ is
$H^2$-smooth away from the crack tip, boundary of defect, and possible
irregular points of external boundary, for detail see 
\cite[Section~2]{KK/00}. 

If $\varepsilon\searrow+0$, similarly to \eqref{2.4} there exists the unique
solution of the {\em homogeneous problem}: Find
$u^0\in K(\Omega\setminus\Gamma)$ such that 
for all $v\in K(\Omega\setminus\Gamma)$ 
\begin{equation}\label{2.7}
\begin{split}
&\int_{\Omega\setminus\Gamma} \mu (\nabla u^0)^\top \nabla(v -u^0) \,dx
\ge\int_{\Gamma_N} g(v -u^0) \,dS_x,
\end{split}
\end{equation}
which implies the boundary value problem: 
\begin{subequations}\label{2.8}
\begin{equation}\label{2.8a}
-\Delta u^0 =0\quad\text{in $\Omega\setminus\Gamma$},
\end{equation}
\begin{equation}\label{2.8b}
u^0 =0\quad\text{on $\Gamma_D$},\quad 
\mu {\textstyle\frac{\partial u^0}{\partial n}}
=g\quad\text{on $\Gamma_N$},
\end{equation}
\begin{equation}\label{2.8c}
\begin{split}
&\bigl[\!\!\bigl[ {\textstyle\frac{\partial u^0}{\partial n}} \bigr]\!\!\bigr]
=0,\quad [\![u^0]\!]\ge0,\quad 
{\textstyle\frac{\partial u^0}{\partial n}}\le0,\quad
{\textstyle\frac{\partial u^0}{\partial n}}
[\![u^0]\!]=0 \quad\text{on $\Gamma$},
\end{split}
\end{equation}
\begin{equation}\label{2.8d}
\begin{split}
&\bigl[\!\!\bigl[ {\textstyle\frac{\partial u^0}{\partial n}} \bigr]\!\!\bigr]
=0,\quad [\![u^0]\!]=0 \quad\text{on $\partial\omega_\varepsilon(x^0)$}.
\end{split}
\end{equation}
\end{subequations}
We note that \eqref{2.8d} is written here for comparison with \eqref{2.5d}, and
it implies that the solution $u^0$ is $C^\infty$-smooth in
$B_\varepsilon(x^0)\supset \omega_\varepsilon(x^0)$ compared to
$u^{(\omega,\varepsilon,x^0,\delta)}$. 

Using Green's formulae separately in $(\Omega\setminus\Gamma)\setminus
\omega_\varepsilon(x^0)$ and in $\omega_\varepsilon(x^0)$, 
the variational inequality \eqref{2.7} can be transformed into an 
equivalent variational inequality depending on the parameter $\delta$: 
\begin{equation}\label{2.9}
\begin{split}
&\int_{\Omega\setminus\Gamma} \mu \chi^\delta_{{}_{\omega_\varepsilon(x^0)}}
(\nabla u^0)^\top \nabla(v -u^0) \,dx
\ge\int_{\Gamma_N} g(v -u^0) \,dS_x\\
&-(1-\delta)\int_{\partial\omega_\varepsilon(x^0)} 
\mu {\textstyle\frac{\partial u^0}{\partial n}} (v -u^0) \,dS_x
\quad\text{for all $v\in K(\Omega\setminus\Gamma)$}.
\end{split}
\end{equation}
The left hand side of \eqref{2.9} has the same operator as \eqref{2.4}, 
this fact will be used in Section~\ref{sec3} for asymptotic analysis of the 
solution $u^{(\omega,\varepsilon,x^0,\delta)}$. 

\section{Topology asymptotic analysis}\label{sec3}

To examine the heterogeneous state \eqref{2.4} in comparison with the
homogeneous one \eqref{2.7} in an explicit way,  we rely on small defects, thus
passing $\varepsilon\searrow+0$  leads to the first order asymptotic
analysis. 
First, for the solution of the state problem we obtain a two-scale asymptotic
expansion, which is related to Green functions. 
For this reason we apply the singular perturbation theory and endow it with
variational arguments. 
With its help, second, we provide a topology sensitivity of the geometry dependent
objective functions representing the mode-III stress intensity factor (SIF)
and the strain energy release rate (SERR) which are the primary physical
characteristics of fracture. 

\subsection{Asymptotic analysis of the solution}\label{sec3.1}

We start with the inner asymptotic expansion of the solution $u^0$ of the 
homogeneous variational inequality \eqref{2.7}, which is $C^\infty$-smooth 
in the ball $B_R(x^0)$ of the radius $R<\min\{r,{\rm dist}(x^0,\partial\Omega)\}$. 
We recall that $r$ is the distance of the defect center $x^0$ from the crack
tip at the origin $0$. 
Due to \eqref{2.8a}, we have the representation (see e.g. \cite[Section~3]{Il'/92}):
\begin{equation}\label{3.1}
\begin{split}
&u^0(x) =u^0(x^0) +\nabla u^0(x^0)^\top (x-x^0) +U_{{}^{x^0}} (x)
\quad\text{for $x\in B_R(x^0)$},\\
&\int_{-\pi}^\pi U_{{}^{x^0}}\,d\theta_0 =\int_{-\pi}^\pi
U_{{}^{x^0}}{\textstyle\frac{x-x^0}{\rho_0}}\,d\theta_0 =0,\quad
U_{{}^{x^0}} ={\rm O}\bigl(\rho_0^2\bigr),\; \nabla U_{{}^{x^0}} ={\rm O}(\rho_0).
\end{split}
\end{equation}
From \eqref{3.1} we infer the expansion of the traction
\begin{equation}\label{3.2}
\begin{split}
&{\textstyle\frac{\partial u^0}{\partial n}} =\nabla u_0(x^0)^\top n
+{\textstyle\frac{\partial U_{{}^{x^0}}}{\partial n}},\quad 
{\textstyle\frac{\partial U_{{}^{x^0}}}{\partial n}} ={\rm O}(\varepsilon) 
\quad\text{on $\partial\omega_\varepsilon(x^0)$}
\end{split}
\end{equation}
which will be used further for expansion of the right hand side in \eqref{2.9}. 

Moreover, to compensate the ${\rm O}(1)$-asymptotic term 
$\nabla u_0(x^0)^\top n$ in \eqref{3.2}, 
we will need to construct a boundary layer near $\partial\omega_\varepsilon(x^0)$. 
For this task, we stretch the coordinates as $y=\frac{x-x^0}{\varepsilon}$
which implies the diffeomorphic map $\omega_\varepsilon(x^0)
\mapsto\omega_1(0)\subset B_1(0)$. 
In the following, the stretched coordinates $y=(y_1,y_2)^\top 
=|y| (\cos\theta_0,\sin\theta_0)^\top$ refer always to the infinite domain. 
In the whole $\mathbb{R}^2$ we introduce the weighted Sobolev space 
\begin{equation*}
\begin{split}
&H^1_\nu(\mathbb{R}^2) =\{v:\; \nu v,\nabla v\in L^2(\mathbb{R}^2)\},
\quad \nu(y) ={\textstyle\frac{1}{|y|\ln|y|}} \quad\text{in
$\mathbb{R}^2\setminus B_2(0)$},
\end{split}
\end{equation*}
with the weight $\nu\in L^\infty(\mathbb{R}^2)$ due to the weighted Poincare
inequality in exterior domains, see \cite{AGG/97}. 
In this space, excluding constant polynomials $\mathbb{P}_0$,
we state the following auxiliary result, 
which is closely related to the generalized polarization tensors considered 
in \cite[Section~3]{AK/04}.

\begin{lemma}\label{lem3.1}
There exists the unique solution of the following variational problem: 
Find $w\in \bigl( H^1_\nu(\mathbb{R}^2) \setminus\mathbb{P}_0
\bigr)^2$, $w=(w_1,w_2)^\top(y)$, such that
\begin{equation}\label{3.3}
\begin{split}
&\int_{\mathbb{R}^2}\! \chi^\delta_{{}_{\omega_1(0)}}\!
\nabla w_i^\top \nabla v \,dy =(1-\delta) \!\int_{\partial \omega_1(0)}\!\!\!
n_i v\,dS_y \quad \text{for all $v\in H^1_\nu(\mathbb{R}^2)$},
\end{split}
\end{equation}
for $i=1,2$, which satisfies the Laplace equation in $\mathbb{R}^2
\setminus\partial\omega_1(0)$ and the
following transmission boundary conditions across $\partial\omega_1(0)$: 
\begin{equation}\label{3.4}
\begin{split}
&{\textstyle\frac{\partial w}{\partial n}}|_{\partial\omega_1(0)^+}
-\delta {\textstyle\frac{\partial w}{\partial n}}|_{\partial\omega_1(0)^-} 
=-(1-\delta)n,\; 
w|_{\partial\omega_1(0)^+} -w|_{\partial\omega_1(0)^-} =0.
\end{split}
\end{equation}
After rescaling, the far-field representation holds 
\begin{equation}\label{3.5}
\begin{split}
&w({\textstyle\frac{x-x^0}{\varepsilon}}) ={\textstyle\frac{\varepsilon}{2\pi}}
A_{(\omega,\delta)} {\textstyle\frac{x-x^0}{\rho_0^2}}
+W(x) \quad\text{for $x\in \mathbb{R}^2
\setminus B_\varepsilon(x^0)$},\\
&\int_{-\pi}^\pi W\,d\theta_0 =\int_{-\pi}^\pi W
{\textstyle\frac{x-x^0}{\rho_0}} \,d\theta_0 =0,\quad
W ={\rm O}\bigl( \bigl({\textstyle\frac{\varepsilon}{\rho_0}}\bigr)^2 \bigr), 
\;\nabla W ={\rm O}\bigl( {\textstyle\frac{\varepsilon^2}{\rho_0^3}} \bigr),
\end{split}
\end{equation}
where the dipole matrix $A_{(\omega,\delta)} \in{\rm
Sym}(\mathbb{R}^{2\times2})$ has entries ($i,j=1,2$):
\begin{equation}\label{3.6}
\begin{split}
&(A_{(\omega,\delta)})_{ij} =(1-\delta) \Bigl\{ \delta_{ij}\, {\rm
meas}_2(\omega_1(0)) +\int_{\partial\omega_1(0)}\!\!\! w_in_j\,dS_y  \Bigr\}. 
\end{split}
\end{equation}
Moreover, $A_{(\omega,\delta)} \in{\rm Spd}(\mathbb{R}^{2\times2})$ if
$\delta\in[0,1)$  and ${\rm meas}_2(\omega_1(0))>0$. 
\end{lemma}

\begin{proof}
The existence of a solution to \eqref{3.3} up to a free constant follows from
the results of \cite{AGG/97}. 
Following \cite[Lemma~3.2]{FV/89}, below we prove the far-field pattern 
\eqref{3.6} in representation \eqref{3.5}. 

For this reason, we split $\mathbb{R}^2$ in the far-field 
$\mathbb{R}^2\setminus B_1(0)$ and the near-field $B_1(0)$. 
Since $w$ from \eqref{3.3} solves the Laplace equation, 
in the far-field it exhibits the outer asymptotic expansion 
\begin{equation}\label{3.7}
\begin{split}
&w(y) ={\textstyle\frac{1}{2\pi}} A_{(\omega,\delta)}
{\textstyle\frac{y}{|y|^2}} +W(y) \quad\text{for $y\in \mathbb{R}^2
\setminus B_1(0)$},\\ 
&\int_{-\pi}^\pi W\,d\theta_0 =\int_{-\pi}^\pi W
{\textstyle\frac{y}{|y|}} \,d\theta_0 =0,\quad
W ={\rm O}\bigl( \bigl({\textstyle\frac{1}{|y|}}\bigr)^2 \bigr), 
\;\nabla W ={\rm O}\bigl( \bigl({\textstyle\frac{1}{|y|}}\bigr)^3 \bigr), 
\end{split}
\end{equation}
which implies \eqref{3.5} after rescaling $y=\frac{x-x^0}{\varepsilon}$. 

In the near-field, we apply the second Green formula for $i,j=1,2$, 
\begin{equation*}
\begin{split}
&0 =\int_{B_1(0)} \chi^\delta_{{}_{\omega_1(0)}} \{\Delta w_i y_j -w_i \Delta
y_j\}\,dy 
=\int_{\partial B_1(0)} \bigl\{ {\textstyle\frac{\partial w_i}{\partial |y|}}
y_j -w_i {\textstyle\frac{\partial y_j}{\partial |y|}} \bigr\} \,dS_y\\ 
&-\int_{\partial\omega_1(0)} \bigl\{ \bigl[ {\textstyle\frac{\partial
w_i}{\partial n}}|_{\partial\omega_1(0)^+} -\delta {\textstyle\frac{\partial
w_i}{\partial n}}|_{\partial\omega_1(0)^-} \bigr]y_j -(1-\delta) w_i
{\textstyle\frac{\partial y_j}{\partial n}} \bigr\} \,dS_y,
\end{split}
\end{equation*}
and substitute here the transmission conditions \eqref{3.4} to derive that 
\begin{equation}\label{3.8}
\begin{split}
&-\int_{\partial B_1(0)} \bigl\{ {\textstyle\frac{\partial w_i}{\partial |y|}}
-w_i\bigr\} {\textstyle\frac{y_j}{|y|}} \,dS_y
=(1-\delta) \int_{\partial\omega_1(0)} \bigl\{ n_i y_j +w_i n_j \bigr\} \,dS_y.
\end{split}
\end{equation}
We apply to \eqref{3.8} the divergence theorem 
\begin{equation*}
\begin{split}
&\int_{\partial\omega_1(0)} n_i y_j \,dS_y =\int_{\omega_1(0)} y_{j,i} \,dy
=\delta_{ij}\, {\rm meas}_2(\omega_1(0))
\end{split}
\end{equation*}
and substitute \eqref{3.7} to calculate the integral over $\partial B_1(0)$ as
\begin{equation*}
\begin{split}
&-\int_{\partial B_1(0)} \bigl\{ {\textstyle\frac{\partial w_i}{\partial |y|}}
-w_i\bigr\} {\textstyle\frac{y_j}{|y|}} \,dS_y ={\textstyle\frac{1}{\pi}} \int_{-\pi}^\pi
(A_{(\omega,\delta)})_{ik} {\textstyle\frac{y_k}{|y|}}
{\textstyle\frac{y_j}{|y|}} \,d\theta_0 = (A_{(\omega,\delta)})_{ij}, 
\end{split}
\end{equation*}
which implies \eqref{3.6}. 

We now prove the symmetry and positive definiteness properties of
$A_{(\omega,\delta)}$. 
Inserting $v=w_j$, $j=1,2$, into \eqref{3.3} we have
\begin{equation*}
\begin{split}
&\int_{\mathbb{R}^2} \chi^\delta_{{}_{\omega_1(0)}}
\nabla w_i^\top \nabla w_j \,dy =(1-\delta) \!\int_{\partial \omega_1(0)}\!\!\!
n_i w_j\,dS_y =(1-\delta) \!\int_{\partial \omega_1(0)}\!\!\! n_j w_i\,dS_y,
\end{split}
\end{equation*}
hence, the symmetry $(A_{(\omega,\delta)})_{ij} =(A_{(\omega,\delta)})_{ji}$ for
$i,j=1,2$ in \eqref{3.6}.
For arbitrary $z\in\mathbb{R}^2$, from \eqref{3.3} we have 
\begin{equation*}
\begin{split}
&0\le\int_{\mathbb{R}^2} \chi^\delta_{{}_{\omega_1(0)}}
|\nabla (z_1w_1 +z_2w_2)|^2 \,dy =(1-\delta) \sum_{i,j=1}^2  \!\int_{\partial
\omega_1(0)}\!\!\! w_iz_in_jz_j\,dS_y.
\end{split}
\end{equation*}
Henceforth, multiplying \eqref{3.6} with $z_iz_j$ and summing the result over
$i,j=1,2$, it follows that
\begin{equation*}
\begin{split}
&\sum_{i,j=1}^2(A_{(\omega,\delta)})_{ij}z_iz_j =(1-\delta) \Bigl\{
|z|^2\, {\rm meas}_2(\omega_1(0)) +\sum_{i,j=1}^2 \int_{\partial \omega_1(0)}\!\!\!
w_iz_in_jz_j\,dS_y \Bigr\}\\
&\ge (1-\delta) |z|^2\, {\rm meas}_2(\omega_1(0))>0,
\end{split}
\end{equation*}
if $1-\delta>0$ and ${\rm meas}_2(\omega_1(0))>0$. 
This completes the proof. 
\end{proof}

It is important to comment on the transmission conditions \eqref{3.4} 
in relation to the stiffness parameter $\delta>0$. 
On the one hand, for $\delta\searrow+0$ implying that $\omega_1(0)$ is a hole,
conditions \eqref{3.4} split as 
\begin{equation}\label{3.9}
\begin{split}
&w^-=w^+\quad\text{on $\partial\omega_1(0)^-$},\quad 
{\textstyle\frac{\partial w}{\partial n}}^+ =-n
\quad\text{on $\partial\omega_1(0)^+$},
\end{split}
\end{equation}
where the indexes $\pm$ mark the traces of the functions in \eqref{3.9}  
 at $\partial\omega_1(0)^\pm$, respectively. 
Henceforth, to find $A_{(\omega,\delta)}$ in \eqref{3.6} instead of
\eqref{3.3}, it suffices to solve the exterior problem under the Neumann
condition \eqref{3.9}: Find $w\in \bigl( H^1_\nu(\mathbb{R}^2
\setminus\omega_1(0)) \setminus\mathbb{P}_0
\bigr)^2$ such that for $i=1,2$
\begin{equation*}
\begin{split}
&\int_{\mathbb{R}^2\setminus\omega_1(0)} 
\nabla w_i^\top \nabla v \,dy =- \!\int_{\partial \omega_1(0)}\!\!\!
n_i v\,dS_y \quad \text{for all $v\in H^1_\nu(\mathbb{R}^2
\setminus\omega_1(0))$}.
\end{split}
\end{equation*}
In this case, $A_{(\omega,\delta)}$ is called the virtual mass, or added mass
matrix according to \cite{PS/51}. 

On the other hand, for $\delta\nearrow+\infty$ implying that 
$\omega_1(0)$ is a rigid inclusion, conditions \eqref{3.4} read 
\begin{equation}\label{3.10}
\begin{split}
&{\textstyle\frac{\partial w}{\partial n}}^- =-n
\quad\text{on $\partial\omega_1(0)^-$},\quad 
w^+ =w^-\quad\text{on $\partial\omega_1(0)^+$}.
\end{split}
\end{equation}
In this case, \eqref{3.3} is split in the interior Neumann problem in
$\omega_1(0)$, and the exterior Dirichlet problem in $\mathbb{R}^2
\setminus\omega_1(0)$. 
The respective $A_{(\omega,\delta)}$ is called the polarization matrix in 
\cite{PS/51}. 

Thus, we have the following. 

\begin{corollary}\label{cor3.1}
The auxiliary problem \eqref{3.3} under the transmission boundary conditions
\eqref{3.4} describes the general case of inclusions of varying stiffness, 
and it accounts for holes (hard obstacles in acoustics) under the Neumann condition
\eqref{3.9} as well as rigid inclusions (soft obstacles in acoustics) under
the Dirichlet condition \eqref{3.10} as the limit cases of the stiffness
parameter $\delta\searrow+0$ and $\delta\nearrow+\infty$, respectively. 
\end{corollary}

With the help of the boundary layer $w$ constructed in Lemma~\ref{lem3.1} we
can represent the first order asymptotic term in the expansion of the perturbed
solution $u^{(\omega,\varepsilon,x^0,\delta)}$ as $\varepsilon\searrow+0$ 
given in the following theorem. 

\begin{theorem}\label{theo3.1}
The solution $u^{(\omega,\varepsilon,x^0,\delta)}\in K(\Omega\setminus\Gamma)$
of the heterogeneous problem \eqref{2.4}, the solution $u^0$ of the homogeneous  
problem \eqref{2.7}, 
and the rescaled solution $w^\varepsilon(x) :=w({\textstyle\frac{x-x^0}{\varepsilon}})$ 
of \eqref{3.3} 
admit the uniform asymptotic representation for $x\in\Omega\setminus\Gamma$: 
\begin{equation}\label{3.11}
\begin{split}
&u^{(\omega,\varepsilon,x^0,\delta)}(x) =u^0(x) 
+\varepsilon \nabla u^0(x^0)^\top w^\varepsilon(x)
\eta_{{}_{\Gamma_D}}(x) +Q(x),
\end{split}
\end{equation}
where  $\eta_{{}_{\Gamma_D}}$ is a smooth cut-off function which is equal to one
except in a neighborhood of the Dirichlet boundary $\Gamma_D$ on which 
$\eta_{{}_{\Gamma_D}}=0$. 
The residual $Q\in H(\Omega\setminus\Gamma)$ and $w^\varepsilon$ exhibit 
the following asymptotic behavior as $\varepsilon\searrow+0$:
\begin{equation}\label{3.12}
\begin{split}
&\|Q\|_{{}_{H^1(\Omega\setminus\Gamma)}} ={\rm O}(\varepsilon^2),
\quad\text{$w^\varepsilon ={\rm O}(\varepsilon)$ far away from $\omega_\varepsilon(x^0)$.}
\end{split}
\end{equation}
\end{theorem}

\begin{proof}
Since $[\![w^\varepsilon]\!]=0$ on $\Gamma_\infty$, we can substitute 
$v =u^0 +\varepsilon
\nabla u^0(x^0)^\top w^\varepsilon \eta_{{}_{\Gamma_D}}\in K(\Omega\setminus\Gamma)$ 
in \eqref{2.4}, and $v =u^{(\omega,\varepsilon,x^0,\delta)} -\varepsilon
\nabla u^0(x^0)^\top w^\varepsilon \eta_{{}_{\Gamma_D}}\in K(\Omega\setminus\Gamma)$ 
in \eqref{2.9} as the test functions, which yields two inequalities. 
Summing them together and dividing by $\mu$ we get
\begin{equation}\label{3.14}
\begin{split}
&\int_{\Omega\setminus\Gamma}\! 
\chi^\delta_{{}_{\omega_\varepsilon(x^0)}}\!
\nabla (u^{(\omega,\varepsilon,x^0,\delta)} -u^0)^\top \nabla Q \,dx
\le (1-\delta) \int_{\partial \omega_\varepsilon(x^0)}\! 
{\textstyle\frac{\partial u^0}{\partial n}} Q\,dS_x,
\end{split}
\end{equation}
where $Q :=u^{(\omega,\varepsilon,x^0,\delta)} -u^0(x) -\varepsilon \nabla
u^0(x^0)^\top w^\varepsilon \eta_{{}_{\Gamma_D}}\in H(\Omega\setminus\Gamma)$ 
is defined according to \eqref{3.11}. 

After rescaling $y={\textstyle\frac{x-x^0}{\varepsilon}}$, with the help of the
Green formula in $\Omega\setminus\Gamma$, from \eqref{3.3} we obtain the
following variational equation written in the bounded domain for
$w^\varepsilon_i(x) :=w_i({\textstyle\frac{x-x^0}{\varepsilon}})$, $i=1,2$: 
\begin{equation}\label{3.13}
\begin{split}
&\int_{\Omega\setminus\Gamma} \chi^\delta_{{}_{\omega_\varepsilon(x^0)}}
(\nabla w^\varepsilon_i)^\top \nabla v \,dx ={\textstyle\frac{1-\delta}{\varepsilon}}
\int_{\partial \omega_\varepsilon(x^0)}\! n_i v\,dS_x\\
&+\int_{\Gamma_N} {\textstyle\frac{\partial w^\varepsilon_i}{\partial n}} v\,dS_x
-\int_{\Gamma} {\textstyle\frac{\partial w^\varepsilon_i}{\partial n}} [\![v]\!]\,dS_x
\quad\text{for all $v\in H(\Omega\setminus\Gamma)$}. 
\end{split}
\end{equation}
Now inserting $v=Q$ into \eqref{3.13} after multiplication by the 
vector $\varepsilon\nabla u^0(x^0)$ and subtracting it from \eqref{3.14} 
results in the following residual estimate 
\begin{equation*}
\begin{split}
&\int_{\Omega\setminus\Gamma} 
\chi^\delta_{{}_{\omega_\varepsilon(x^0)}} |\nabla Q|^2 \,dx
\le (1-\delta) \int_{\partial \omega_\varepsilon(x^0)}\! 
\bigl( {\textstyle\frac{\partial u^0}{\partial n}} -\nabla
u^0(x^0)^\top n\bigr)Q\,dS_x\\
&-\varepsilon \int_{\Gamma_N} {\textstyle\frac{\partial}{\partial n}}
\bigl( \nabla u^0(x^0)^\top \!w^\varepsilon \bigr) Q\,dS_x
+\varepsilon \int_{\Gamma} {\textstyle\frac{\partial}{\partial n}}
\bigl( \nabla u^0(x^0)^\top \!w^\varepsilon \bigr)
[\![Q]\!]\,dS_x\\
&+\varepsilon \int_{{\rm supp}(1-\eta_{{}_{\Gamma_D}})} 
\chi^\delta_{{}_{\omega_\varepsilon(x^0)}} 
\nabla \bigl( \nabla u^0(x^0)^\top \!w^\varepsilon 
(1-\eta_{{}_{\Gamma_D}}) \bigr)^\top \nabla Q \,dx.
\end{split}
\end{equation*}
We apply here the expansion  \eqref{3.2} at $\partial \omega_\varepsilon(x^0)$ 
which implies that 
$\|\nabla Q\|_{{}_{L^2(\Omega\setminus\Gamma)}} ={\rm O}(\varepsilon^2)$, 
hence the first estimate in \eqref{3.12}. 
The pointwise estimate $w^\varepsilon ={\rm O}(\varepsilon)$ holds far away 
from $\omega_\varepsilon(x^0)$ due to \eqref{3.5}. 
The proof is complete.
\end{proof}

In the following sections we apply Theorem~\ref{theo3.1} for the topology
sensitivity of the objective functions which depend on both the crack $\Gamma$
and the defect $\omega_\varepsilon(x^0)$. 

\subsection{Topology sensitivity of SIF-function}\label{sec3.2}

We start with the notation of {\em stress intensity factor} (SIF). 
At the crack tip $0$, where the stress is concentrated, from \eqref{2.5a} and
\eqref{2.5c} we infer the inner asymptotic expansion (compare to \eqref{3.1}) 
for $x\in B_R(0)\setminus\Gamma$ with $R =\min\{r,{\rm dist}(0,\partial\Omega)\}$:
\begin{equation}\label{3.15}
\begin{split}
&u^{(\omega,\varepsilon,x^0,\delta)}(x)
=u^{(\omega,\varepsilon,x^0,\delta)}(0) +{\textstyle\frac{1}{\mu}}
\sqrt{\textstyle\frac{2}{\pi}} {c}^{(\omega,\varepsilon,x^0,\delta)}_\Gamma
\sqrt{\rho} \sin{\textstyle\frac{\theta}{2}} +U(x),\\
&\int_{-\pi}^\pi U\,d\theta =\int_{-\pi}^\pi U
\bigl( \cos{\textstyle\frac{\theta}{2}}, \sin{\textstyle\frac{\theta}{2}}
\bigr)^\top \,d\theta =0,\quad U ={\rm O}(\rho), 
\; \nabla U ={\rm O}(1).
\end{split}
\end{equation}
In the fracture literature, the factor
${c}^{(\omega,\varepsilon,x^0,\delta)}_\Gamma$ in \eqref{3.15} is called SIF, here
due to the mode-III crack in the anti-plane setting of the spatial fracture
problem. 
The SIF characterizes the main singularity at the crack tip. 
Moreover, the inequality conditions \eqref{2.5c} require necessarily 
\begin{equation}\label{3.16}
\begin{split}
&{c}^{(\omega,\varepsilon,x^0,\delta)}_\Gamma\ge0. 
\end{split}
\end{equation}

For the justification of \eqref{3.15} and \eqref{3.16} we refer to \cite{KKT/10,KK/08}, 
where the homogeneous nonlinear model with rectilinear crack \eqref{2.8} was considered. 
Indeed, this asymptotic result is stated by the method of separation of variables 
locally in the neighborhood $B_R(0)\setminus\Gamma$ away from the inhomogeneity 
$\omega_\varepsilon(x^0)$. 
Here the governing equations \eqref{2.5a} and \eqref{2.5c} for $u^0$ coincide 
with the equations \eqref{2.8a} and \eqref{2.8c} for the solution 
$u^{(\omega,\varepsilon,x^0,\delta)}$ of the inhomogeneous problem. 
Therefore, the inner asymptotic expansions \eqref{3.15} of 
$u^{(\omega,\varepsilon,x^0,\delta)}$ and \eqref{3.25} of $u^0$ 
are similar and differ only by constant parameters (the SIF). 
For a respective mechanical confirmation see \cite{Leb/00}.  

Below we sketch a Saint--Venant estimate proving the bound of $\nabla U$ in \eqref{3.15}. 
Since $\Delta U=0$, then $U$ is a harmonic function which is infinitely differentiable 
in $B_R(0)\setminus\Gamma$, and integrating by parts we derive for $t\in(0,R)$:
\begin{equation*}
\begin{split}
&I(t) :=\int_{B_t(0)\setminus\Gamma} |\nabla U|^2 \,dx 
=\int_{\partial B_t(0)} {\textstyle\frac{\partial U}{\partial\rho}} U \,dS_x 
-\int_{B_t(0)\cap\Gamma} {\textstyle\frac{\partial U}{\partial n}} [\![U]\!] \,dS_x\\
&\le \int_{-\pi}^\pi {\textstyle\frac{\partial U}{\partial\rho}} U \,t d\theta
\le \int_{-\pi}^\pi \bigl( {\textstyle\frac{t}{2}} ({\textstyle\frac{\partial 
U}{\partial\rho}})^2 +{\textstyle\frac{1}{2t}} U^2 \bigr) \,t d\theta
\le \int_{-\pi}^\pi \bigl( {\textstyle\frac{t}{2}} ({\textstyle\frac{\partial 
U}{\partial\rho}})^2 +{\textstyle\frac{1}{2t}} ({\textstyle\frac{\partial 
U}{\partial\theta}})^2 \bigr) \,t d\theta\\
&={\textstyle\frac{t}{2}} \int_{\partial B_t(0)} |\nabla U|^2 \,dS_x 
={\textstyle\frac{t}{2}} {\textstyle\frac{d}{dt}} I(t).
\end{split}
\end{equation*}
Here we have used, consequently: conditions \eqref{2.5c} justifying that 
at $B_t(0)\cap\Gamma$
\begin{equation*}
\begin{split}
&0 ={\textstyle\frac{\partial u^{(\omega,\varepsilon,x^0,\delta)}}{\partial n}}
[\![u^{(\omega,\varepsilon,x^0,\delta)}]\!] ={\textstyle\frac{\partial U}{\partial n}} 
\bigl( {\textstyle\frac{2}{\mu}}
\sqrt{\textstyle\frac{2}{\pi}} {c}^{(\omega,\varepsilon,x^0,\delta)}_\Gamma
\sqrt{\rho} + [\![U]\!] \bigr) \le {\textstyle\frac{\partial U}{\partial n}}  [\![U]\!] 
\end{split}
\end{equation*}
due to \eqref{3.15} and \eqref{3.16}, Young's and Wirtinger's inequalities, 
and the co-area formula. 
Integrating this differential inequality results in the estimate 
$I(t)\le ({\textstyle\frac{t}{R}})^2 I(R)$, which implies $I(t) ={\rm O}(t^2)$ 
and follows $\nabla U ={\rm O}(1)$ in \eqref{3.15}. 

From a mathematical viewpoint, the factor in \eqref{3.15} can be determined
in the dual space of $H(\Omega\setminus\Gamma)$ through the so-called weight function, which we introduce next. 
While the existence of a weight function is well known for the linear crack problem,  
e.g. in \cite[Chapter~6]{MNP/00}, here we modify it for the underlying nonlinear problem. 
In fact, the modified weight function $\zeta$ provides formula \eqref{3.24} 
representing the SIF. 

Let $\eta(\rho)$ be a smooth cut-off function supported in the disk 
$B_{2R}(0)\subset\Omega$, $\eta\equiv1$ in $B_R(0)$, and $R>r$, where $r>0$
stands always for the distance to the defect. 
With the help of the cut-off function we extend in $\Omega$ the tangential vector 
$\tau$ from the crack $\Gamma$ by the vector 
\begin{equation}\label{3.17}
\begin{split}
&V(x) :=\tau\,\eta(\rho),\quad \tau=(1,0)^\top
\end{split}
\end{equation}
which is  used further for the shape sensitivity in \eqref{3.36} 
following the velocity method commonly adopted in shape optimization \cite{SZ/92}.
Using the notation of matrices for
\begin{equation}\label{3.18}
\begin{split}
&D(V) :={\rm div}(V) \,{\rm Id} -{\textstyle\frac{\partial V}{\partial x}}
-{\textstyle\frac{\partial V}{\partial x}}^\top 
\in {\rm Sym} (\mathbb{R}^{2\times2}),
\end{split}
\end{equation}
where ${\rm Id}$ means the identity matrix, the coincidence set 
\begin{equation*}
\begin{split}
&\Xi :=\{x\in\Gamma:\; [\![u^0]\!]=0\}
\end{split}
\end{equation*}
and the 'square-root' function 
$S(x):= \sqrt{\textstyle\frac{2}{\pi}} \sqrt{\rho} \sin{\textstyle\frac{\theta}{2}}$, 
we formulate the auxiliary variational problem: 
Find $\xi\in H(\Omega\setminus\Gamma)$ such that 
\begin{equation}\label{3.19}
\begin{split}
&[\![\xi]\!] =[\![V^\top \nabla S]\!]\;\text{on $\Xi$},
\quad\int_{\Omega\setminus\Gamma} \nabla \xi^\top \nabla v \,dx
=-\int_{\Omega\setminus\Gamma} \nabla S^\top D(V) \nabla v \,dx\\
&\text{for all $v\in H(\Omega\setminus\Gamma)$ with $[\![v]\!] =0$ on $\Xi$},
\end{split}
\end{equation}
where $V^\top\nabla S =-{\textstyle\frac{1}{\sqrt{2\pi}}}
{\textstyle\frac{1}{\sqrt{\rho}}} \sin{\textstyle\frac{\theta}{2}} \eta$ 
is the directional derivative of $S$ with respect to $V$, 
and $[\![V^\top \nabla S]\!] =-\sqrt{\textstyle\frac{2}{\pi \rho}} \eta$. 

\begin{remark}\label{rem3.1}
Due to the inhomogeneous condition stated at $\Xi$ in \eqref{3.19}, 
to provide $[\![\xi]\!]\in H^{1/2}_{00}(\Gamma)$ we assume that 
the coincidence set $\Xi$ where $[\![u^0]\!]=0$ is separated from the crack tip, 
i.e. $0\not\in\overline{\Xi}$. 
For example, this assumption is guaranteed for the stress intensity factor 
${c}^0_\Gamma>0$ (see the definition of ${c}^0_\Gamma$ in \eqref{3.25}) 
when the crack is open in the vicinity. 
Otherwise, if the crack is closed $[\![u^0]\!]\equiv 0$ in a neighborhood 
$[-C,0]\times\{0\} \subset\Gamma$ of the crack tip $(0,0)$, 
then the crack problem can be restated for the crack tip $(-C,0)$. 
\end{remark}

In order to get the strong formulation we use the following identities 
in the right-hand side of \eqref{3.19}: 
\begin{equation*}
\begin{split}
&{\rm div} \bigl( \nabla S^\top D(V) \bigr)
={\rm div}(V) \Delta S -\Delta V^\top \nabla S\\ 
&-2\bigl( \nabla V_1^\top \nabla S_{,1} +\nabla V_2^\top \nabla S_{,2} \bigr) 
=-\Delta (V^\top \nabla S)
\end{split}
\end{equation*}
where we have applied $\Delta S=0$ in $\Omega\setminus\Gamma$, and 
\begin{equation*}
\begin{split}
&\bigl( \nabla S^\top D(V) \bigr) n
=0
=-{\textstyle\frac{\partial}{\partial n}} (V^\top \nabla S) 
\quad\text{on $\Gamma^\pm$}
\end{split}
\end{equation*}
due to $V_2=0$, ${\textstyle\frac{\partial V}{\partial n}}=0$, 
and ${\textstyle\frac{\partial S}{\partial n}}=0$, 
recalling that ${\textstyle\frac{\partial}{\partial n}} 
=-{\textstyle\frac{1}{\rho}} {\textstyle\frac{\partial}{\partial\theta}}$ 
at $\Gamma^\pm$, as $\theta =\pm\pi$. 
Henceforth, after integration of \eqref{3.19} by parts, 
the unique solution of \eqref{3.19} satisfies the mixed Dirichlet--Neumann problem: 
\begin{subequations}\label{3.20}
\begin{equation}\label{3.20a}
-\Delta \xi =-\Delta \bigl( V^\top \nabla S \bigr)
\quad\text{in $\Omega\setminus\Gamma$},
\end{equation}
\begin{equation}\label{3.20b}
\xi =0\quad\text{on $\Gamma_D$},\quad 
{\textstyle\frac{\partial \xi}{\partial n}} =0 \quad\text{on $\Gamma_N$},
\end{equation}
\begin{equation}\label{3.20c}
\begin{split}
&[\![\xi]\!] =[\![V^\top \nabla S]\!],\quad \bigl[\!\!\bigr[ 
{\textstyle\frac{\partial \xi}{\partial n}} \bigl]\!\!\bigr] 
=\bigl[\!\!\bigr[ {\textstyle\frac{\partial}{\partial n}}
\bigl( V^\top \nabla S \bigr) \bigl]\!\!\bigr] =0\quad\text{on $\Xi$},\\
&{\textstyle\frac{\partial \xi}{\partial n}}
={\textstyle\frac{\partial}{\partial n}}
\bigl( V^\top \nabla S \bigr) =0\quad\text{on $(\Gamma\setminus\Xi)^\pm$}.
\end{split}
\end{equation}
\end{subequations}

From \eqref{3.19} and \eqref{3.20} we define the {\em weight function} 
(here $t>0$ small) 
\begin{equation}\label{3.21}
\begin{split}
&\zeta :=\xi -V^\top \nabla S \in L^2(\Omega\setminus\Gamma)\cap
H^1\bigl( (\Omega\setminus\Gamma)\setminus B_t(0) \bigr),
\end{split}
\end{equation}
which is a non-trivial singular solution of the homogeneous problem
\begin{subequations}\label{3.22}
\begin{equation}\label{3.22a}
-\Delta \zeta =0 \quad\text{in $\Omega\setminus\Gamma$},
\end{equation}
\begin{equation}\label{3.22b}
\zeta =0\quad\text{on $\Gamma_D$},\quad 
{\textstyle\frac{\partial \zeta}{\partial n}} =0 \quad\text{on $\Gamma_N$},
\end{equation}
\begin{equation}\label{3.22c}
\begin{split}
&[\![\zeta]\!] =\bigl[\!\!\bigr[ {\textstyle\frac{\partial \zeta}{\partial n}} 
\bigl]\!\!\bigr] =0\;\text{on $\Xi$},\quad 
{\textstyle\frac{\partial \zeta}{\partial n}} =0
\;\text{on $(\Gamma\setminus\Xi)^\pm$}.
\end{split}
\end{equation}
\end{subequations}
For comparison, for the linear crack problem the coincidence set $\Xi=\emptyset$ 
and the mixed Dirichlet--Neumann problem \eqref{3.22} turns into the 
homogeneous Neumann problem for the weight function $\zeta$. 
From \eqref{3.21} it follows that 
\begin{equation}\label{3.23}
\begin{split}
&\zeta(x) ={\textstyle\frac{1}{\sqrt{2\pi}}}
{\textstyle\frac{1}{\sqrt{\rho}}} \sin{\textstyle\frac{\theta}{2}} 
+\xi(x)\quad\text{for $x\in B_R(0)\setminus\{0\}$},
\end{split}
\end{equation}
which is useful in the following. 

\begin{lemma}\label{lem3.2}
For $0\not\in\overline{\Xi}$ providing solvability of the problem \eqref{3.19}, 
the stress intensity factor ${c}^{(\omega,\varepsilon,x^0,\delta)}_\Gamma$ from
\eqref{3.15} and \eqref{3.16} can be determined by the following integral formula 
\begin{equation}\label{3.24}
\begin{split}
&{c}^{(\omega,\varepsilon,x^0,\delta)}_\Gamma =\max\bigl\{0, 
\int_{\Gamma_N} g\zeta\,dS_x -\mu \int_{\partial\omega_\varepsilon(x^0)}
\bigl[\!\!\bigl[ {\textstyle\frac{\partial
u^{(\omega,\varepsilon,x^0,\delta)}}{\partial n}} \bigr]\!\!\bigr] \zeta\,dS_x\\
&+\mu \int_{\Xi} {\textstyle\frac{\partial \zeta}{\partial n}} 
[\![u^{(\omega,\varepsilon,x^0,\delta)}]\!] \,dS_x 
-\mu \int_{\Gamma\setminus\Xi}
{\textstyle\frac{\partial u^{(\omega,\varepsilon,x^0,\delta)}}{\partial n}} 
[\![\zeta]\!] \,dS_x \bigr\}
\end{split}
\end{equation}
with the weight function $\zeta$ defined in \eqref{3.19} and \eqref{3.21} 
together with the properties \eqref{3.22} and \eqref{3.23}. 
\end{lemma}

\begin{proof}
Using the second Green formula in $(\Omega\setminus\Gamma)\setminus B_t(0)$
with small $t\in(0,t_0)$, from \eqref{2.5} and \eqref{3.20} we derive that 
\begin{equation*}
\begin{split}
&0 =\int_{(\Omega\setminus\Gamma)\setminus B_t(0)}
\{\Delta \zeta u^{(\omega,\varepsilon,x^0,\delta)} -\zeta \Delta
u^{(\omega,\varepsilon,x^0,\delta)}\} \,dx =-{\textstyle\frac{1}{\mu}}
\int_{\Gamma_N} g\zeta\,dS_x\\ 
&+\int_{\partial\omega_\varepsilon(x^0)}
\bigl[\!\!\bigl[ {\textstyle\frac{\partial
u^{(\omega,\varepsilon,x^0,\delta)}}{\partial n}} \bigr]\!\!\bigr] \zeta\,dS_x
-\int_{\partial B_t(0)} \bigl\{ 
{\textstyle\frac{\partial \zeta}{\partial \rho}}
u^{(\omega,\varepsilon,x^0,\delta)} -\zeta
{\textstyle\frac{\partial u^{(\omega,\varepsilon,x^0,\delta)}}{\partial \rho}}
\bigr\} \,dS_x\\
&-\int_{\Gamma\setminus B_t(0)} \bigl\{ 
{\textstyle\frac{\partial \zeta}{\partial n}}
[\![u^{(\omega,\varepsilon,x^0,\delta)}]\!] -[\![\zeta]\!]
{\textstyle\frac{\partial u^{(\omega,\varepsilon,x^0,\delta)}}{\partial n}}
\bigr\} \,dS_x.
\end{split}
\end{equation*}
In the latter integral over $\Gamma\setminus B_t(0)$, 
the first summand vanishes at $(\Gamma\setminus\overline{\Xi})\setminus B_t(0)$, 
and the second summand is zero at $\Xi\setminus B_t(0)$ due to \eqref{3.22c}. 

For fixed $\varepsilon$ and $t\searrow+0$, 
since the coincidence set is detached from the crack tip: 
there exists $C>0$ such that $B_C(0)\cap\Xi =\emptyset$, 
then the integral over $\Xi\setminus B_t(0)$ is uniformly bounded: 
\begin{equation*}
\begin{split}
&\int_{\Xi\setminus B_t(0)} {\textstyle\frac{\partial \zeta}{\partial n}} 
[\![u^{(\omega,\varepsilon,x^0,\delta)}]\!] \,dS_x 
=\int_{\Xi\setminus B_C(0)} {\textstyle\frac{\partial \zeta}{\partial n}} 
[\![u^{(\omega,\varepsilon,x^0,\delta)}]\!] \,dS_x ={\rm O}(1) 
\;\text{as $t\searrow+0$}.
\end{split}
\end{equation*}
This integral is well defined because the solution $\zeta$ 
of the mixed Dirichlet--Neumann problem \eqref{3.19} exhibits 
the square-root singularity (see e.g. \cite{MNP/00} and references therein), 
hence ${\textstyle\frac{\partial \zeta}{\partial n}}$ 
has the one-over-square-root singularity which is integrable, and 
$[\![u^{(\omega,\varepsilon,x^0,\delta)}]\!]$ is $H^{3/2}$-smooth in 
$\Xi\setminus B_C(0)$. 
The $H^2$-regularity of the solution to the nonlinear crack problem 
up to the crack faces except the crack vicinity is proved rigorously 
e.g. in \cite{BKK/00} with the shift technique. 
Similarly, the integral over $(\Gamma\setminus\overline{\Xi})\setminus B_t(0)$ 
is uniformly bounded: 
\begin{equation*}
\begin{split}
&\int_{(\Gamma\setminus\Xi)\setminus B_t(0)}
{\textstyle\frac{\partial u^{(\omega,\varepsilon,x^0,\delta)}}{\partial n}} 
[\![\zeta]\!] \,dS_x
=\int_{(\Gamma\setminus\Xi)\setminus B_{t_0}(0)}
{\textstyle\frac{\partial u^{(\omega,\varepsilon,x^0,\delta)}}{\partial n}} 
[\![\zeta]\!] \,dS_x\\ 
&+\int_t^{t_0}
{\textstyle\frac{1}{\rho}} {\textstyle\frac{\partial U}{\partial\theta}}
\bigl( \sqrt{\textstyle\frac{2}{\pi \rho}} +[\![\xi]\!] \bigr) \,d\rho 
={\rm O}(1)\;\text{as $t\searrow+0$}
\end{split}
\end{equation*}
due to the representations \eqref{3.15} and \eqref{3.23}, 
and ${\textstyle\frac{1}{\rho}} {\textstyle\frac{\partial U}{\partial\theta}} 
={\rm O}(1)$ according to \eqref{3.15}.

The former integral over $\partial B_t(0)$ can be calculated 
by plugging the representations \eqref{3.15} and \eqref{3.23} here 
\begin{equation*}
\begin{split}
&-\int_{\partial B_t(0)} \bigl\{ 
{\textstyle\frac{\partial \zeta}{\partial \rho}}
u^{(\omega,\varepsilon,x^0,\delta)} -\zeta
{\textstyle\frac{\partial u^{(\omega,\varepsilon,x^0,\delta)}}{\partial \rho}}
\bigr\} \,dS_x
= {\textstyle\frac{{c}^{(\omega,\varepsilon,x^0,\delta)}_\Gamma}{\mu\pi}} 
\int_{-\pi}^\pi \sin^2({\textstyle\frac{\theta}{2}}) \,d\theta\\ 
&-t \int_{-\pi}^\pi \bigl\{ {\textstyle\frac{\partial \xi}{\partial \rho}} 
\bigl( u^{(\omega,\varepsilon,x^0,\delta)}(0) 
+{\textstyle\frac{{c}^{(\omega,\varepsilon,x^0,\delta)}_\Gamma}{\mu}}
\sqrt{\textstyle\frac{2t}{\pi}} \sin{\textstyle\frac{\theta}{2}} +U \bigr)\\ 
&-\xi\bigl( 
{\textstyle\frac{{c}^{(\omega,\varepsilon,x^0,\delta)}_\Gamma}{\mu\sqrt{2\pi t}}} 
\sin{\textstyle\frac{\theta}{2}} 
+{\textstyle\frac{\partial U}{\partial \rho}}\bigr) \bigr\} \,d\theta
={\textstyle\frac{1}{\mu}}
{c}^{(\omega,\varepsilon,x^0,\delta)}_\Gamma +{\rm O}(\sqrt{t})
\end{split}
\end{equation*}
which holds true due to $\xi ={\rm O}(1)$, 
$\frac{\partial \xi}{\partial \rho} ={\rm O}(\frac{1}{\sqrt{t}})$ 
(similarly to \eqref{3.15}), using $dS_x =t d\theta$ and \eqref{3.15} 
for $U(\rho,\theta)$ as $\rho=t$ and $\theta\in(-\pi,\pi)$. 
Therefore, passing $t\searrow+0$ and accounting for \eqref{3.16} we have proven
formula \eqref{3.24}.
\end{proof}

Next, using Theorem~\ref{theo3.1} we expand the right hand side of \eqref{3.24} 
in $\varepsilon\searrow+0$ and derive the main result of this section.

\begin{theorem}\label{theo3.2}
For $0\not\in\overline{\Xi}$, 
the SIF ${c}^{(\omega,\varepsilon,x^0,\delta)}_\Gamma$ of the heterogeneous
problem \eqref{2.4} given in \eqref{3.24} admits the following asymptotic
representation 
\begin{equation}\label{3.27}
\begin{split}
&{c}^{(\omega,\varepsilon,x^0,\delta)}_\Gamma =\max\bigl\{0, 
\int_{\Gamma_N} g\zeta\,dS_x -\varepsilon^2 \mu
\nabla u^0(x^0)^\top A_{(\omega,\delta)} \nabla \zeta(x^0)\\ 
&+\mu \int_{\Xi} {\textstyle\frac{\partial \zeta}{\partial n}} 
[\![u^{(\omega,\varepsilon,x^0,\delta)}]\!] \,dS_x 
-\mu \int_{\Gamma\setminus\Xi}
{\textstyle\frac{\partial u^{(\omega,\varepsilon,x^0,\delta)}}{\partial n}} 
[\![\zeta]\!] \,dS_x +{\rm Res} \bigr\},\\
&{\rm Res} ={\rm O}(\varepsilon^3),
\end{split}
\end{equation}
where $A_{(\omega,\delta)}$ is the dipole matrix and 
$\nabla \zeta(x^0) = {\textstyle\frac{1}{2\sqrt{2\pi}}} r^{-3/2} \bigl(
-\sin{\textstyle\frac{3\phi}{2}}, \cos{\textstyle\frac{3\phi}{2}} \bigr)^\top +{\rm O}(r^{-1/2})$ 
at the defect center $x^0=r(\cos\phi,\sin\phi)^\top$. 
\end{theorem}

\begin{proof}
To expand the integral over $\partial\omega_\varepsilon(x^0)$ 
in the right hand side of \eqref{3.24} as
$\varepsilon\searrow+0$, we substitute here the expansion \eqref{3.11} of the
solution $u^{(\omega,\varepsilon,x^0,\delta)}$ which implies 
\begin{equation}\label{3.28}
\begin{split}
&\int_{\partial\omega_\varepsilon(x^0)}
\bigl[\!\!\bigl[ {\textstyle\frac{\partial
u^{(\omega,\varepsilon,x^0,\delta)}}{\partial n}} \bigr]\!\!\bigr] \zeta\,dS_x
=\varepsilon \nabla u^0(x^0)^\top \!\!\!\int_{\partial\omega_\varepsilon(x^0)}
\bigl[\!\!\bigl[ {\textstyle\frac{\partial 
w^\varepsilon}{\partial n}} \bigr]\!\!\bigr] \zeta\,dS_x +{\rm O}(\varepsilon^3).
\end{split}
\end{equation}

Below we apply to the right hand side of \eqref{3.28} the expansion \eqref{3.5}
of the boundary layer $w^\varepsilon$ and the inner asymptotic expansion 
of $\zeta$, which is a $C^\infty$-function in the near field of $x^0$, 
written similarly to \eqref{3.1} as 
\begin{equation}\label{3.29}
\begin{split}
&\zeta(x) =\zeta(x^0) +\nabla \zeta(x^0)^\top (x-x^0) +Z(x)
\quad\text{for $x\in B_R(x^0)$},\\
&\int_{-\pi}^\pi Z \,d\theta_0 =\int_{-\pi}^\pi Z
{\textstyle\frac{x-x^0}{\rho_0}} \,d\theta_0 =0,\quad
Z ={\rm O}\bigl(\rho_0^2\bigr), \;\nabla Z ={\rm O}(\rho_0).
\end{split}
\end{equation}
Next inserting \eqref{3.5} and \eqref{3.29} into the second Green formula in
$B_\varepsilon(x^0)$, 
\begin{equation*}
\begin{split}
&\int_{\partial\omega_\varepsilon(x^0)} \bigl\{ \bigl[\!\!\bigl[ 
{\textstyle\frac{\partial w^\varepsilon}{\partial n}} \bigr]\!\!\bigr]
\zeta +(1-\delta) w^\varepsilon 
{\textstyle\frac{\partial \zeta}{\partial n}} \bigr\} \,dS_x
=\int_{\partial B_\varepsilon(x^0)} \bigl\{
{\textstyle\frac{\partial \zeta}{\partial \rho_0}} w^\varepsilon
-\zeta {\textstyle\frac{\partial w^\varepsilon}{\partial \rho_0}} \bigr\} \,dS_x,
\end{split}
\end{equation*}
we estimate its terms as follows. 
The divergence theorem provides   
\begin{equation*}
\begin{split}
&\int_{\partial\omega_\varepsilon(x^0)} 
w^\varepsilon {\textstyle\frac{\partial \zeta}{\partial n}} \,dS_x
=\int_{\partial\omega_\varepsilon(x^0)} 
n^\top \nabla \zeta(x^0) \,dS_x +{\rm O}(\varepsilon^2)\\
&=\int_{\omega_\varepsilon(x^0)} 
(\nabla w^\varepsilon)^\top \nabla \zeta(x^0) \,dx +{\rm O}(\varepsilon^2)
={\rm O}(\varepsilon^2),
\end{split}
\end{equation*}
and we calculate analytically the integral over $\partial B_\varepsilon(x^0)$
as 
\begin{equation*}
\begin{split}
&\int_{\partial B_\varepsilon(x^0)}\!\!\! \bigl\{
{\textstyle\frac{\partial \zeta}{\partial \rho_0}} w^\varepsilon
-\zeta {\textstyle\frac{\partial w^\varepsilon}{\partial \rho_0}} \bigr\} \,dS_x 
={\textstyle\frac{\varepsilon}{\pi}} \int_{-\pi}^\pi\!\!\! \nabla
\zeta(x^0)^\top {\textstyle\frac{x-x^0}{\rho_0}} A_{(\omega,\delta)}
{\textstyle\frac{x-x^0}{\rho_0}} \,d\theta_0  +{\rm O}(\varepsilon^2)\\
&=\varepsilon A_{(\omega,\delta)} \nabla \zeta(x^0) +{\rm O}(\varepsilon^2).
\end{split}
\end{equation*}
Therefore, we obtain the asymptotic expansion 
\begin{equation}\label{3.30}
\begin{split}
&\int_{\partial\omega_\varepsilon(x^0)}
\bigl[\!\!\bigl[ {\textstyle\frac{\partial w^\varepsilon}{\partial n}} \bigr]\!\!\bigr]
\zeta \,dS_x
=\varepsilon A_{(\omega,\delta)} \nabla \zeta(x^0) +{\rm O}(\varepsilon^2).
\end{split}
\end{equation}
Inserting \eqref{3.28} and \eqref{3.30} into \eqref{3.24} and using \eqref{3.16} 
it yields \eqref{3.27}. 
Finally, the value of $\nabla \zeta(x^0)$ can be estimated analytically from \eqref{3.23}, 
while $\xi$ has the ${\rm O}(\rho^{1/2})$-singularity similar to \eqref{3.15}, hence $\nabla \xi(x^0) ={\rm O}(r^{-1/2})$. 
This completes the proof. 
\end{proof}

As the corollary of Lemma~\ref{lem3.2} and Theorem~\ref{theo3.2} we find the SIF
of the solution $u^0\in K(\Omega\setminus\Gamma)$ of the homogeneous problem
\eqref{2.7}, which is the limit case of the heterogeneous problem as
$\varepsilon\searrow+0$. 
Namely, similar to \eqref{3.15} and \eqref{3.16} we have the inner asymptotic expansion 
\begin{equation}\label{3.25}
\begin{split}
&u^0(x) =u^0(0) +{\textstyle\frac{1}{\mu}}
\sqrt{\textstyle\frac{2}{\pi}} {c}^0_\Gamma
\sqrt{\rho} \sin{\textstyle\frac{\theta}{2}} +U^0(x)
\quad\text{for $x\in B_R(0)\setminus\Gamma$},\\
&\int_{-\pi}^\pi U^0\,d\theta =\int_{-\pi}^\pi U^0
\bigl( \cos{\textstyle\frac{\theta}{2}}, \sin{\textstyle\frac{\theta}{2}}
\bigr) \,d\theta =0,\quad U^0 ={\rm O}(\rho),\; \nabla U^0 ={\rm O}(1) 
\end{split}
\end{equation}
with the reference SIF ${c}^0_\Gamma\ge0$ determined by the formula 
\begin{equation}\label{3.26}
\begin{split}
&{c}^0_\Gamma =\max\bigl\{0, \int_{\Gamma_N} g\zeta\,dS_x \bigr\},
\end{split}
\end{equation}
where we have used the complementarity conditions 
${\textstyle\frac{\partial \zeta}{\partial n}} [\![u^0]\!] 
={\textstyle\frac{\partial u^0}{\partial n}} [\![\zeta]\!] =0$ hold at $\Gamma$ 
due to \eqref{3.22c} and \eqref{2.8c} providing 
${\textstyle\frac{\partial u^0}{\partial n}}=0$ at $\Gamma\setminus\overline{\Xi}$.

In the following we derive an interpretation of Theorem~\ref{theo3.2} 
from the point of view of shape-topological control. 

We parametrize the crack growth by means of the position of the crack tip
along the fixed path $x_2=0$ as 
\begin{equation*}
\begin{split}
&\Gamma_\infty(t) :=\{x\in\mathbb{R}^2: x_1<t, x_2=0\},\quad 
\Gamma(t) :=\Gamma_\infty(t)\cap\Omega,
\end{split}
\end{equation*}
such that $\Gamma=\Gamma(0)$ in this notation. 
Formula \eqref{3.24} defines the optimal value function depending on both
$\Gamma(t)$ and $\omega_\varepsilon(x^0)$ 
\begin{equation}\label{3.31}
\begin{split}
&J_{\rm SIF}: \mathbb{R} \times\Theta \times\mathbb{R}_+ 
\times(\Omega\setminus\Gamma) \times\overline{\mathbb{R}}_+\mapsto \mathbb{R}_+,\\
&(t,\omega,\varepsilon,x^0,\delta)\mapsto J_{\rm SIF}(\Gamma(t),\omega_\varepsilon(x^0))
:={c}^{(\omega,\varepsilon,x^0,\delta)}_{\Gamma(t)}
\end{split}
\end{equation}
and satisfying the consistency condition
$\omega_\varepsilon(x^0) \subset B_\varepsilon(x^0) \subset\Omega\setminus\Gamma(t)$. 
From the physical point of view, 
the reason of \eqref{3.31} is to control the SIF of the crack $\Gamma(t)$ 
by means of the defect $\omega_\varepsilon(x^0)$. 
The homogeneous reference state implies 
\begin{equation}\label{3.32}
\begin{split}
&J_{\rm SIF}(\Gamma(t),\emptyset) ={c}^0_{\Gamma(t)}.
\end{split}
\end{equation}
For fixed $\Gamma(0)=\Gamma$, formula \eqref{3.27} proves the topology
sensitivity of $J_{\rm SIF}$ from \eqref{3.31} and \eqref{3.32} with respect to
diminishing the defect $\omega_\varepsilon(x^0)$ as $\varepsilon\searrow+0$. 

In the following section we introduce another geometry dependent objective
function inherently related to fracture, namely, the strain energy release rate
(SERR). 
We lead its first order
topology sensitivity analysis using the result of Theorem~\ref{theo3.2}. 
The first order asymptotic term provides us with the respective {\em topological
derivative}, see in \cite{HK/11} a generalized concept of topological
derivatives suitable for fracture due to cracks. 

\subsection{Topological derivative of SERR-function}\label{sec3.3}

The widely used Griffith criterion of fracture declares that a crack starts to
grow when its SERR attains a critical value (the material parameter of fracture
resistance). 
Therefore, decreasing SERR would arrest the incipient crack growth, 
while increasing SERR, conversely, will affect its rise. 
This gives us practical motivation of the topological derivative of 
the SERR objective function, which we construct below. 

After substitution of the solution $u^{(\omega,\varepsilon,x^0,\delta)}$ of the
heterogeneous problem \eqref{2.4}, the reduced energy functional \eqref{2.3}
implies 
\begin{equation}\label{3.34}
\begin{split}
&\Pi(\Gamma(t),\omega_\varepsilon(x^0))\\
&={\textstyle\frac{1}{2}} \int_{\Omega\setminus\Gamma(t)} \mu
\chi^\delta_{{}_{\omega_\varepsilon(x^0)}}
|\nabla u^{(\omega,\varepsilon,x^0,\delta)}|^2\,dx 
-\int_{\Gamma_N} gu^{(\omega,\varepsilon,x^0,\delta)} \,dS_x.
\end{split}
\end{equation}
The derivative of $\Pi$ in \eqref{3.34} with respect to $t$, taken with the
minus sign, is called {\em strain energy release rate} (SERR) and defines the
optimal value function similar to \eqref{3.31} as
\begin{equation}\label{3.35}
\begin{split}
&J_{\rm SERR}: \mathbb{R} \times\Theta \times\mathbb{R}_+ 
\times(\Omega\setminus\Gamma) \times\overline{\mathbb{R}}_+\mapsto \mathbb{R}_+,\\
&(t,\omega,\varepsilon,x^0,\delta)\mapsto 
J_{\rm SERR}(\Gamma(t),\omega_\varepsilon(x^0)) :=
-{\textstyle\frac{d}{dt}} \Pi(\Gamma(t),\omega_\varepsilon(x^0)).
\end{split}
\end{equation}
It admits the equivalent representations 
(see \cite{HK/11,KK/00,KS/99,KS/00,Kov/02} for detail): 
\begin{equation}\label{3.36}
\begin{split}
&J_{\rm SERR} =-{\textstyle\frac{1}{2}}
\int_{\Omega\setminus\Gamma(t)} \mu
\chi^\delta_{{}_{\omega_\varepsilon(x^0)}}
(\nabla u^{(\omega,\varepsilon,x^0,\delta)})^\top D(V) 
\nabla u^{(\omega,\varepsilon,x^0,\delta)} \,dx\\
&=\lim_{R\searrow+0} I_R,\quad\text{where}\quad  
I_R :=\mu \int_{\partial B_R((t,0))} \bigl\{ {\textstyle\frac{1}{2}}
\bigl(V^\top{\textstyle\frac{x}{\rho}}\bigr) 
|\nabla u^{(\omega,\varepsilon,x^0,\delta)})|^2\\ 
&-(V^\top\nabla u^{(\omega,\varepsilon,x^0,\delta)}) \bigl(
{\textstyle\frac{x^\top}{\rho}} \nabla u^{(\omega,\varepsilon,x^0,\delta)}\bigr)
\bigr\} \,dS_x. 
\end{split}
\end{equation}
The key issue is that from \eqref{3.36} we derive the following expression 
\begin{equation}\label{3.37}
\begin{split}
&J_{\rm SERR}(\Gamma(t),\omega_\varepsilon(x^0)) 
={\textstyle\frac{1}{2\mu}} \bigl(
{c}^{(\omega,\varepsilon,x^0,\delta)}_{\Gamma(t)} \bigr)^2 \ge0.
\end{split}
\end{equation}
Indeed, from the local asymptotic expansion \eqref{3.15} written 
at the crack tip $(t,0)$ it follows 
\begin{equation*}
\begin{split}
&\nabla u^{(\omega,\varepsilon,x^0,\delta)}
={\textstyle\frac{1}{\mu \sqrt{2\pi R}}}\, 
{c}^{(\omega,\varepsilon,x^0,\delta)}_{\Gamma(t)} 
\bigl( -\sin{\textstyle\frac{\theta}{2}}, \cos{\textstyle\frac{\theta}{2}}\bigr)^\top
+\nabla U\quad\text{on $\partial B_R((t,0))$}.
\end{split}
\end{equation*}
Plugging this expression into the invariant integral $I_R$ in \eqref{3.36}, 
due to $|\nabla U| ={\rm O}(1)$, $V=(1,0)$ and 
${\textstyle\frac{x^\top}{\rho}} =(\cos\theta, \sin\theta)^\top$ 
at $\partial B_R((t,0))$, we calculate
\begin{equation*}
\begin{split}
I_R &=\mu \int_{-\pi}^\pi \bigl\{ {\textstyle\frac{1}{2}} \cos\theta
{\textstyle\frac{1}{2\pi R \mu^2}} 
({c}^{(\omega,\varepsilon,x^0,\delta)}_{\Gamma(t)})^2 
 +\sin^2({\textstyle\frac{\theta}{2}}) {\textstyle\frac{1}{2\pi R \mu^2}} 
({c}^{(\omega,\varepsilon,x^0,\delta)}_{\Gamma(t)})^2\\ 
&+{\rm O} \Bigl({\textstyle\frac{|\nabla U|}{\sqrt{R}}}\Bigr) 
\bigr\} \,R d\theta ={\textstyle\frac{1}{2\mu}} \bigl(
{c}^{(\omega,\varepsilon,x^0,\delta)}_{\Gamma(t)} \bigr)^2 
+{\rm O}(\sqrt{R}). 
\end{split}
\end{equation*}
Passing $R\searrow+0$ it follows \eqref{3.37}. 
Now, the substitution of expansion \eqref{3.27} in \eqref{3.37} proves
directly the asymptotic model of SERR as $\varepsilon\searrow+0$ given 
next. 

\begin{theorem}\label{theo3.3}
For $0\not\in\overline{\Xi}$, the strain energy release rate at the tip 
of the crack  $\Gamma=\Gamma(0)$ admits the following asymptotic representation 
when diminishing the defect $\omega_\varepsilon(x^0)$:
\begin{equation}\label{3.38}
\begin{split}
&J_{\rm SERR}(\Gamma,\omega_\varepsilon(x^0)) 
={\textstyle\frac{1}{2\mu}} \bigl( {c}^0_{\Gamma} \bigr)^2 -\varepsilon^2
{c}^0_{\Gamma} \nabla u^0(x^0)^\top A_{(\omega,\delta)} \nabla \zeta(x^0)\\ 
&+{c}^0_{\Gamma} \int_{\Xi\setminus\Xi^\varepsilon} 
{\textstyle\frac{\partial \zeta}{\partial n}} 
[\![u^{(\omega,\varepsilon,x^0,\delta)} -u^0]\!] \,dS_x 
-{c}^0_{\Gamma} \int_{\Xi^\varepsilon\setminus\Xi}
{\textstyle\frac{\partial (u^{(\omega,\varepsilon,x^0,\delta)} -u^0)}{\partial n}} 
[\![\zeta]\!] \,dS_x\\ 
&+{\rm Res},\quad \text{${\rm Res}={\rm O}(\varepsilon^3)$ 
and ${\rm Res}\ge0$ if ${c}^0_{\Gamma}=0$,}
\end{split}
\end{equation}
where the perturbed coincidence set is determined by 
\begin{equation*}
\begin{split}
&\Xi^\varepsilon :=\{x\in\Gamma:\; [\![u^{(\omega,\varepsilon,x^0,\delta)}]\!]=0\}. 
\end{split}
\end{equation*}  
The reference $J_{\rm SERR}(\Gamma,\emptyset) =\frac{1}{2\mu} ({c}^0_{\Gamma})^2$ 
implies SERR for the homogeneous state $u^0$ without defect, 
$A_{(\omega,\delta)}$ is the dipole matrix, and the gradient 
$\nabla \zeta(x^0) = {\textstyle\frac{1}{2\sqrt{2\pi}}} r^{-3/2} \bigl(
-\sin{\textstyle\frac{3\phi}{2}}, \cos{\textstyle\frac{3\phi}{2}}
\bigr)^\top +{\rm O}(r^{-1/2})$ at the defect center $x^0=r(\cos\phi,\sin\phi)^\top$.

Moreover, if the coincidence sets are continuous such that 
${\rm meas}_1(\Xi^\varepsilon\setminus\Xi)\searrow+0$ 
and ${\rm meas}_1(\Xi\setminus\Xi^\varepsilon)\searrow+0$ 
as $\varepsilon\searrow+0$, then 
the first asymptotic term in \eqref{3.38} provides the topological derivative 
\begin{equation}\label{3.39}
\begin{split}
&\lim_{\varepsilon\searrow+0} {\textstyle
\frac{J_{\rm SERR}(\Gamma,\omega_\varepsilon(x^0))  
-J_{\rm SERR}(\Gamma,\emptyset)}{\varepsilon^2}} 
=-{c}^0_{\Gamma} \nabla u^0(x^0)^\top A_{(\omega,\delta)} \nabla \zeta(x^0). 
\end{split}
\end{equation}
\end{theorem}

\begin{proof}
To derive \eqref{3.38} we square \eqref{3.27} and \eqref{3.26}. 
Then we use, first, that 
\begin{equation}\label{3.40}
\begin{split}
&\int_{\Xi} {\textstyle\frac{\partial \zeta}{\partial n}} 
[\![u^{(\omega,\varepsilon,x^0,\delta)}]\!] \,dS_x 
=\int_{\Xi\setminus\Xi^\varepsilon} {\textstyle\frac{\partial \zeta}{\partial n}}
[\![u^{(\omega,\varepsilon,x^0,\delta)}]\!] \,dS_x\\ 
&=\int_{\Xi\setminus\Xi^\varepsilon} 
{\textstyle\frac{\partial \zeta}{\partial n}} 
[\![u^{(\omega,\varepsilon,x^0,\delta)} -u^0]\!] \,dS_x 
\end{split}
\end{equation}
holds due to $[\![u^{(\omega,\varepsilon,x^0,\delta)}]\!] =0$ at $\Xi^\varepsilon$ 
and $[\![u^0]\!] =0$ at $\Xi$. 
Second, the equality 
\begin{equation}\label{3.41}
\begin{split}
&\int_{\Gamma\setminus\Xi}
{\textstyle\frac{\partial u^{(\omega,\varepsilon,x^0,\delta)}}{\partial n}} 
[\![\zeta]\!] \,dS_x 
=\int_{\Xi^\varepsilon\setminus\Xi}
{\textstyle\frac{\partial u^{(\omega,\varepsilon,x^0,\delta)}}{\partial n}} 
[\![\zeta]\!] \,dS_x\\ 
&=\int_{\Xi^\varepsilon\setminus\Xi}
{\textstyle\frac{\partial (u^{(\omega,\varepsilon,x^0,\delta)} -u^0)}{\partial n}} 
[\![\zeta]\!] \,dS_x
\end{split}
\end{equation}
holds due to 
${\textstyle\frac{\partial u^{(\omega,\varepsilon,x^0,\delta)}}{\partial n}} =0$
at $\Gamma\setminus\overline{\Xi^\varepsilon}$ and 
${\textstyle\frac{\partial u^0}{\partial n}} =0$ 
at $\Gamma\setminus\overline{\Xi}$ according to the 
complementarity conditions \eqref{2.5c} and \eqref{2.8c} and using the identity 
$(\Gamma\setminus\overline{\Xi})\setminus(\Gamma\setminus\overline{\Xi^\varepsilon}) 
=\overline{\Xi^\varepsilon}\setminus\overline{\Xi}$. 

To justify \eqref{3.39} it needs to pass  \eqref{3.40} and \eqref{3.41} 
divided by $\varepsilon^2$ to the limit as $\varepsilon\searrow+0$. 
For this task we employ the assumption that $0\not\in\overline{\Xi}$ 
and the assumption of continuity of the coincidence sets, 
hence $0\not\in\overline{\Xi^\varepsilon}$ for sufficiently small $\varepsilon$. 
Otherwise, $0\in\overline{\Xi^\varepsilon}$ implies 
${c}^{(\omega,\varepsilon,x^0,\delta)}_{\Gamma} =0$ that contradicts to 
the convergence ${c}^{(\omega,\varepsilon,x^0,\delta)}_{\Gamma} 
\to{c}^0_{\Gamma}\not=0$ as $\varepsilon\searrow+0$ 
following from \eqref{3.15} and \eqref{3.25} due to Theorem~\ref{theo3.1}. 
This implies that the sets $\Xi\setminus\overline{\Xi^\varepsilon}$ as well as 
$\Xi^\varepsilon\setminus\overline{\Xi}$ are detached from the crack tip. 
Henceforth, the functions $[\![u^{(\omega,\varepsilon,x^0,\delta)} -u^0]\!]\in 
H^{3/2}(\Xi\setminus\Xi^\varepsilon)$ and 
${\textstyle\frac{\partial (u^{(\omega,\varepsilon,x^0,\delta)} -u^0)}{\partial n}}, 
[\![\zeta]\!] \in L^2(\Xi^\varepsilon\setminus\Xi)$ are smooth here, 
and the following asymptotic estimates hold
\begin{equation*}
\begin{split}
&\int_{\Xi\setminus\Xi^\varepsilon} 
{\textstyle\frac{\partial \zeta}{\partial n}} 
[\![u^{(\omega,\varepsilon,x^0,\delta)} -u^0]\!] \,dS_x 
=\int_{\Xi\setminus\Xi^\varepsilon} \bigl( 
{\textstyle\frac{\partial (\zeta -\zeta^\varepsilon)}{\partial n}} 
+{\textstyle\frac{\partial \zeta^\varepsilon}{\partial n}} \bigr)
[\![u^{(\omega,\varepsilon,x^0,\delta)} -u^0]\!] \,dS_x\\ 
&\le \|{\textstyle\frac{\partial (\zeta -\zeta^\varepsilon)}{\partial n}} 
\|_{H^{1/2}(\Gamma)^\star} \bigl\| [\![u^{(\omega,\varepsilon,x^0,\delta)} 
-u^0]\!] \bigr\|_{H^{1/2}(\Xi\setminus\Xi^\varepsilon)}\\
&+\|{\textstyle\frac{\partial \zeta^\varepsilon}{\partial n}} 
\|_{L^2(\Xi\setminus\Xi^\varepsilon)} \bigl\| [\![u^{(\omega,\varepsilon,x^0,\delta)} 
-u^0]\!] \bigr\|_{L^2(\Xi\setminus\Xi^\varepsilon)} ={\rm o}(\varepsilon^2)
\end{split}
\end{equation*}
where $\zeta^\varepsilon$ is a smooth approximation of $\zeta$ such that 
$ \|{\textstyle\frac{\partial (\zeta -\zeta^\varepsilon)}{\partial n}} 
\|_{H^{1/2}(\Gamma)^\star} ={\rm o}(1)$ and 
$\|{\textstyle\frac{\partial \zeta^\varepsilon}{\partial n}} 
\|_{L^2(\Xi\setminus\Xi^\varepsilon)} ={\rm o}(1)$ as $\varepsilon\searrow+0$, and 
\begin{equation*}
\begin{split}
&\int_{\Xi^\varepsilon\setminus\Xi}
{\textstyle\frac{\partial (u^{(\omega,\varepsilon,x^0,\delta)} -u^0)}{\partial n}} 
[\![\zeta]\!] \,dS_x \le \bigl\| {\textstyle\frac{\partial 
(u^{(\omega,\varepsilon,x^0,\delta)} -u^0)}{\partial n}} 
\bigr\|_{L^2(\Xi^\varepsilon\setminus\Xi)}
\| [\![\zeta]\!] \|_{L^2(\Xi^\varepsilon\setminus\Xi)}  ={\rm o}(\varepsilon^2)
\end{split}
\end{equation*}
provided by Theorem~\ref{theo3.1} and the assumption of the convergence 
${\rm meas}_1(\Xi^\varepsilon\setminus\Xi)\searrow+0$ 
and ${\rm meas}_1(\Xi\setminus\Xi^\varepsilon)\searrow+0$ as $\varepsilon\searrow+0$. 
This proves the limit in \eqref{3.39} and the assertion of the theorem. 
\end{proof}

\section{Discussion}\label{sec4}

In the context of fracture, from Theorem~\ref{theo3.3} we can discuss the
following. 

The Griffith fracture criterion suggests that the crack $\Gamma$ starts to grow
when $J_{\rm SERR} =G_c$ attains the fracture resistance threshold $G_c>0$. 
For incipient growth of the nonlinear crack subject to inequality
${c}^0_{\Gamma}>0$, its arrest necessitates the negative topological derivative
to decrease $J_{\rm SERR}$,
hence positive sign of $\nabla u^0(x^0)^\top A_{(\omega,\delta)} \nabla
\zeta(x^0)$ in \eqref{3.38}. 

The sign and value of the topological derivative depends in semi-analytic
implicit way on the solution $u^0$, trial center $x^0$, shape
$\omega$ and stiffness $\delta$ of the defect. 
The latter two parameters enter the topological derivative through the dipole
matrix $A_{(\omega,\delta)}$. 
In Appendix~\ref{A} we present explicit values of the dipole matrix for the
specific cases of the ellipse shaped holes and inclusions. 
This describes also the degenerate case of cracks and thin rigid inclusions 
called anti-cracks. 

\vspace{5mm}\noindent
{\bf Acknowledgment}. 
V.A.~Kovtunenko is supported by the Austrian Science Fund (FWF) 
project P26147-N26 (PION) and partially by NAWI Graz 
and OeAD Scientific \& Technological Cooperation (WTZ CZ 01/2016).\\ 
G.~Leugering is supported by DFG EC 315 "Engineering of Advanced Materials". 
The authors thank J.~Sokolowski for the discussion 
and two referees for the remarks allowing to substantially improve the original manuscript.

\appendix
\section{Ellipse and crack shaped defects}\label{A}

Let the shape $\omega$ of a defect be ellipsoidal. 
Namely, we consider the ellipse $\omega$ enclosed in the ball $B_1(0)$, which
has the major one and the minor $b\in(0,1]$ semi-axes, where the major axis has
an angle of $\alpha\in[0,2\pi)$ with the $x_1$-axis counted in the
anti-clockwise direction. 

With the rotation matrix $Q(\alpha)$, the dipole matrix for the {\em
elliptic defect} has the form (see \cite{BHV/03,PMMM/12}) 
\begin{subequations}\label{a1}
\begin{equation}\label{a1a}
\begin{split}
&A_{(\omega,\delta)} = Q(\alpha) A_{(\omega',\delta)} Q(\alpha)^\top,\quad 
Q(\alpha) :=\left(\begin{array}{lr}
\cos\alpha &-\sin\alpha\\ \sin\alpha& \cos\alpha
\end{array} \right),
\end{split}
\end{equation}
\begin{equation}\label{a1b}
\begin{split}
&A_{(\omega',\delta)} =\pi(1+b)\left(\begin{array}{lr}
{\textstyle\frac{(1-\delta)b}{1+\delta b}} &0\\ 0&
{\textstyle\frac{(1-\delta)b}{\delta +b}}
\end{array} \right).
\end{split}
\end{equation}
\end{subequations}
Further we consider the limit cases of \eqref{a1b} when the
stiffness parameter $\delta\searrow+0$ and $\delta\nearrow+\infty$, which
correspond to the ellipse shaped holes and rigid inclusions according to
Corollary~\ref{cor3.1}. 

On the one hand, for the {\em elliptic hole} $\omega$, passing
$\delta\searrow+0$ in \eqref{a1b} we obtain the virtual mass, or added mass
matrix 
\begin{equation}\label{a2}
\begin{split}
&A_{(\omega',\delta)} =\pi (1+b) \left(\begin{array}{lr}
b &0\\ 0&1
\end{array} \right),
\end{split}
\end{equation}
which is positive definite. 
In particular, for the {\em straight crack} $\omega$ as $b\searrow+0$,
\eqref{a2} turns in the singular matrix 
\begin{equation}\label{a3}
\begin{split}
&A_{(\omega',\delta)} =\pi \left(\begin{array}{lr}
0 &0\\ 0&1
\end{array} \right).
\end{split}
\end{equation}

On the other hand, for the {\em rigid ellipse} $\omega$, passing
$\delta\nearrow+\infty$ in \eqref{a1b} we obtain the polarization matrix 
\begin{equation}\label{a4}
\begin{split}
&A_{(\omega',\delta)} =\pi (1+b) \left(\begin{array}{lr}
-1 &0\\ 0&-b
\end{array} \right),
\end{split}
\end{equation}
which is negative definite. 
In particular, for the {\em rigid segment} $\omega$ 
as $b\searrow+0$, \eqref{a4} turns in the singular matrix 
\begin{equation}\label{a5}
\begin{split}
&A_{(\omega',\delta)} =\pi \left(\begin{array}{lr}
-1 &0\\ 0&0
\end{array} \right).
\end{split}
\end{equation}

\end{document}